\documentclass[12pt]{amsart}
\usepackage{geometry} % see geometry.pdf on how to lay out the page. There's lots.
\usepackage{amsmath,amsthm,amsfonts,amssymb}
\usepackage{graphics,xy}
\usepackage{xcolor}
\usepackage{graphicx}
\usepackage{ulem}
\usepackage{url}
\usepackage{tikz}
\usepackage[pagebackref=true]{hyperref}
\usepackage[utf8]{inputenc}
\newtheorem{lemma}{Lemma}[section]

\newtheorem{theorem}[lemma]{Theorem}

\newtheorem{corollary}[lemma]{Corollary}

\theoremstyle{definition}
\newtheorem{definition}[lemma]{Definition}
\newtheorem{remark}[lemma]{Remark}

\newcommand{\Z}{\mathbb{Z}}
\newcommand{\R}{\mathbb{R}}
\newcommand{\Q}{\mathbb{Q}}
\newcommand{\C}{\mathbb{C}}
\newcommand{\N}{\mathbb{N}}

\newcommand{\Sk}{\textrm{Sk}}
\newcommand{\ra}{\longrightarrow}

\title[Finiteness of KBSM of $3$-manifolds]{An effective proof of finiteness for Kauffman bracket skein modules}

\author{Giulio Belletti}
\address{
	IRMP, UC Louvain, Chemin du Cyclotron 2, bte L7.01.02, 1348 Louvain-la-Neuve, Belgium}
\email{gbelletti451@gmail.com}
\author{Renaud Detcherry}
\date{} % delete this line to display the current date

\address{Institut de Mathématiques de Bourgogne, UMR 5584 CNRS, Université Bourgogne Franche-Comté, F-2100 Dijon, France}
\email{renaud.detcherry@u-bourgogne.fr}
\date{} % delete this line to display the current date

\begin{document}
\maketitle
	\begin{abstract} We prove a version of the finiteness conjecture for Kauffman bracket skein modules of $3$-manifolds with boundary, which was introduced by the second author in \cite{Det21}. In particular, our methods, which are constructive, give an alternative proof of Witten's finiteness conjecture for the Kauffman bracket skein modules of closed $3$-manifolds, which was originally proved in \cite{GJS19}. Moreover, as a corollary we show that the peripheral ideal of any link is non-empty, answering a question of Frohman, Gelca and Lofaro \cite{FGL02}. 
\end{abstract}
\tableofcontents
\section{Introduction}
\label{sec:intro}
\subsection{Context and statements of the main results}
\label{sec:intro-mainresult}
Kauffman bracket skein modules of compact oriented $3$-manifolds were defined independently by Przytycki \cite{Prz1} and by Turaev \cite{Tur} to generalize the Jones polynomial of links in $S^3$ to links in any $3$-manifold. Kauffman bracket skein modules and their generalizations have been a central object in quantum topology, appearing in many of its facets. However, until fairly recent, the computation of Kauffman bracket skein modules of $3$-manifold was seen as intractable in all but the most simple examples of $3$-manifolds, as they have often infinitely generated \cite{HP95}, and sometimes not even a sum of finitely generated $\Z[A^{\pm 1}]$-modules \cite{Mro11}, \cite{BKSW}.

Kauffman bracket skein modules may be defined over various choices of rings of coefficients. We will denote by $Sk(M)$ the Kauffman bracket skein module over $\Q(A)$ coefficients, and by $Sk(M,\Z[A^{\pm 1}])$ the Kauffman bracket skein module with $\Z[A^{\pm 1}]$ coefficients (see Definition \ref{dfn:KBSMZ} and \ref{dfn:KBSM} for details).

In this context, the finiteness conjecture of Witten (first in print in \cite{Car}) stated that skein modules of closed oriented $3$-manifold are finite dimensional over $\Q(A)$. 

The finiteness conjecture was later proved by Gunningham, Jordan and Safronov, by a stunning proof involving an array of ideas coming from factorization homology, quantum groups and holonomic $D$-modules. We describe succintly their proof, which is non-constructive, in Section \ref{sec:intro-outline}.

In \cite{Det21}, the second author introduced several versions of the finiteness conjecture for compact oriented $3$-manifolds with boundary. The weakest version of the conjecture, positing that the skein module (with $\Q(A)$ coefficients) of a manifold is finite dimensional as a module over the skein module of its boundary, has since been proven for some Seifert manifolds \cite{AF22} and for some families of genus $2$ and $3$ manifolds \cite{KW23}. The strongest version of the conjecture had been already known at the time for torus knots \cite{Mar10} and for $2$-bridge knots and links \cite{Le06} \cite{LT14}.

The purpose of this paper is to give a new constructive proof of the finiteness theorem for Kauffman bracket skein modules of closed $3$-manifolds, and to prove an extension to compact oriented $3$-manifolds with boundary. To precisely state our result, which is similar in spirit to \cite[Conjecture 3.3]{Det21}, we need the following definition:
\begin{definition}
	Let $M$ be a compact manifold with boundary $\partial M$, and let $\overline{\lambda}:=(\lambda_1,\dots,\lambda_n)$ be disjoint curves in $\partial M$. Then $\Q(A)[\overline{\lambda}]$ forms a commutative subring of $\Sk(\partial M)$; denote with $\Q(A,\overline{\lambda})$ the field of fractions of this subring. Then the \emph{localized skein module} of $M$ with respect to the curves $\lambda_1,\dots,\lambda_n$, denoted with $\Sk^{\overline{\lambda}}(M)$, is the module $\Sk(M,\Q(A))\otimes_{\Q(A)[\overline{\lambda}]}\Q(A,\overline{\lambda})$.
\end{definition}

For $\Sigma$ a closed compact oriented surface, we say that a family $\mathcal{C}$ of simple closed curves on $\Sigma$ is a \textit{pants decomposition} of $\Sigma$ is $\mathcal{C}$ is a maximal family of non-trivial, disjoint, pairwise non-parallel simple closed curves on $\Sigma.$ If $\Sigma$ has no component that has genus $0$ or $1,$ then this matches the usual notion of pants decomposition since $\mathcal{C}$ will then decompose $\Sigma$ into pairs of pants. However, for each component $S$ of $\Sigma$ of genus $1,$ a pants decomposition $\mathcal{C}$ of $\Sigma$ contains a single non-trivial simple closed curve in $S,$ and for each component $S$ of genus $0,$ the family $\mathcal{C}$ contains no simple closed curve in $S.$

With this definition, we can state our main theorem: 

\begin{theorem}\label{thm:main}
	Let $M$ be a compact manifold with (possibly empty) boundary $\partial M$, and let $\overline{\lambda}:=(\lambda_1,\dots,\lambda_n)$ be a pants decomposition of $\partial M.$ Then, $\Sk^{\overline{\lambda}}(M)$ is a finitely generated $\Q(A,\lambda_1,\dots,\lambda_n)$-module.
\end{theorem}
\begin{remark}
	When $M$ is closed, the pants decomposition $\overline{\lambda}$ is empty and $Sk^{\overline{\lambda}}(M)$ is simply $Sk(M,\Q(A)).$ Hence, Witten's finiteness conjecture for Kauffman bracket skein modules of closed $3$-manifolds is a special case of Theorem \ref{thm:main}, and we recover the result of \cite{GJS19}. 
\end{remark}
\begin{remark}
	It might appear surprising that Theorem \ref{thm:main} holds with such a weak condition on the set of curves $\lambda_1,\dots,\lambda_n.$ We can get away with such a weak condition because of the localization we considered here, in constrast to Conjecture 3.3 of \cite{Det21}. For instance, we claim that if $\overline{\lambda}$ contains a curve $\lambda_1$ that bounds a disk in $M,$ then $Sk^{\overline{\lambda}}(M)=0.$ Meanwhile, $Sk(D^2\times S^1,\Q(A))$ is for example not finitely generated as a $\Q(A)[m]$-module, where $m$ is the meridian of $D^2\times S^1.$
\end{remark}

While writing this article, the authors were informed that David Jordan and Iordanis Romaidis have a proof of \cite[Conjecture 3.1]{Det21} which uses techniques similar to \cite{GJS19} and which will appear in \cite{JR25}; this conjecture states that for any compact oriented $3$-manifold $M$, $Sk(M)$ is a finitely generated module over $Sk(\partial M)$.

\subsection{Outline of the proof and comparison with the work of \cite{GJS19}}
\label{sec:intro-outline}
As was mentioned above, Theorem \ref{thm:main} recovers as a special case the finite dimensionality of Kauffman bracket skein of $3$-manifolds over $\Q(A),$ which was known as Witten's finiteness conjecture and was first proved by Gunningham, Jordan and Safronov in \cite{GJS19}. We stress that our methods of proof differ greatly from the methods in \cite{GJS19}. To describe their proof succintly, for a closed $3$-manifold $M$ they introduced a new version of the skein module, $Sk^{int}(M),$ which they called the \textit{internal skein module}, which has a natural $U_qsl_2$-action, and such that the invariant part recovers the ordinary skein module $Sk(M).$ They prove that the internal skein algebras of surfaces and the internal skein modules of handlebodies are deformation quantization of their representations variety in $SL_2(\C),$ which allows them, by considering a Heegaard splitting $M=H\underset{\Sigma}{\cup} \overline{H'},$ to relate $Sk^{int}(M)$ to a tensor product of two holonomic $D$-modules, from which the finiteness theorem can be deduced by a theorem of Kashiwara and Schapira.

In constrast, while our proof starts by considering a Heegaard splitting of $M$ (into a handlebody and compression body in our case), this is the only ingredient our proofs have in common. Not only do we work directly with the skein modules and not with the internal (or stated) version, but we do not use at all the relationship between Kauffman bracket skein modules and $SL_2(\C)$-character varieties. Instead, given a Heegaard splitting
$$M=\overline{H}\underset{\Sigma}{\cup}C,$$
our starting point is that the Kauffman bracket skein module $Sk(M)$ is spanned by links in $H.$ Skein modules of handlebodies admit some nice basis $(\varphi_c)$ over $\Q(A),$ parametrized by \textit{admissible colorings} $c\in \Delta$ of a spine of the handlebody, where $\Delta$ is the intersection of an explicit polytope in $\R^{3g-3}$ with a lattice. Moreover, in Section \ref{sec:filtration} we prove a key topological lemma (Lemma \ref{lem:linDep}), which allows us to produce elements of $Sk(\Sigma)$ of a specific kind that vanishes in $Sk^{\overline{\lambda}}(C).$
Then, in Section \ref{sec:curvop} we use the description of the action of $Sk(\Sigma)$ on $Sk(H),$ which was studied in details by the second author and Santharoubane in \cite{DS25}, to provide some linear dependence relations between the vectors $\varphi_c,$ viewed as elements of $Sk^{\overline{\lambda}}(M).$ 
For instance, we get linear recurrence relations of the form
$$R_k(A,A^{c_1},\ldots,A^{c_{3g-3}})\varphi_{c+k\delta_i}+ \ldots + R_{-k}(A,A^{c_1},\ldots,A^{c_{3g-3}})\varphi_{c-k\delta_i}=0,$$
where $R_i$ are polynomials and $\delta_i$ is the vector with $i$-th coordinate $1$ and other coordinates $0.$ 

In Sections \ref{sec:induction1} and \ref{sec:induction2}, we prove that those recurrence relations are sufficient to prove finite dimensionality of $Sk^{\overline{\lambda}}(M).$ We proceed by induction, proving first that the span of all $\varphi_c$ for $c\in \Delta \setminus \underset{i\in I}{\cup} V_i$ is finite dimensional, for some finite family of subspaces $V_i$ of codimension at least $1,$ then proceed to increase the codimension of subspaces in the family $(V_i)_{i\in I}.$ 

The most technical part of our argument is contained in Section \ref{sec:finalproof}. Because some of our recurrence relations break down when $c$ approaches the boundary of the polytope $\Delta,$ we are forced to analyze the boundary of the polytope $\Delta,$  and find recurrence relations adapted to each facet of the polytope; this will be done in Section \ref{sec:polytope-boundary}.

\subsection{Corollaries of our proof of the main theorem}
\label{sec:intro-corollaries}
One remarkable feature of \cite{GJS19} is that it proves Witten's finiteness conjecture not only for the Kauffman bracket skein modules, but for $G$-skein modules of closed manifolds, for any complex reductive group $G.$ As explained in \cite[Section 2.3 and 3.1]{GJS19}, to any ribbon category $\mathcal{C}$ and any $3$-manifold, one can associate a skein module $Sk_{\mathcal{C}}(M);$ in this setting, the $G$-skein module $Sk_G(M)$ is obtained for the choice $\mathcal{C}=Rep_q(G),$ which is a ribbon category associated to $G,$ originally defined by Lustzig \cite[Chapter 32]{Lust:book}, and the Kauffman bracket skein module is a particular case of this construction for $G=SL_2.$

In this paper, we restrict ourselves to the case of the Kauffman bracket skein module, though the spirit of our proof may be applicable to $G$-skein modules more generally. Indeed, a basis of $Sk_G(H)$ for a handlebody $H$ may again be described in terms of colorings of a spine by elements of a certain polytope, and the natural action of $Sk_G(\Sigma)$ on $Sk_G(H)$ should have a very similar form.
However, some of the most technical parts of our proof require a delicate analysis of the admissibility polytope, which would require a lot more complicate combinatorics for general $G.$ We may investigate this in subsequent work.

Now, our methods have some advantage compared to \cite{GJS19}: indeed they are more explicit. The finiteness of skein modules in \cite{GJS19} in the end stems from applying an abstract theorem about tensor product of $D$-modules; in particular it is not constructive. In constrast, with our methods, we get:

\begin{corollary}
	\label{cor:explicit-genset} Let $M$ be a closed $3$-manifold, given by a Heegaard splitting $M=\overline{H}\underset{\Sigma}{\cup}H'.$ Then there exists an algorithm which produces an explicit finite generating set for $Sk(M).$
\end{corollary}
Indeed, our whole proof gives a way of reducing the infinite generating set $(\varphi_c)_{c\in \Delta}$ down to a finite generating set. For detailed comments about how our methods are algorithmic, see Section \ref{sec:corollaries}.

On another note, for $\zeta\in \C,$ and $M$ a $3$-manifold, let us denote by $Sk_{\zeta}(M)$
the Kauffman bracket skein module evaluted at $A=\zeta.$ The finiteness theorem implies that $Sk_{\zeta}(M)$ is finite dimensional over $\C$ when $\zeta$ is transcendental. In constrast, for $\zeta$ a root of unity, $Sk_{\zeta}(M)$ may be infinite dimensional, see \cite[Theorem 2.1]{DKS}. An open question is whether $Sk_{\zeta}(M)$ is finite dimensional whenever $\zeta$ is not a root of unity. As a corollary of our proof, we get the following partial answer:
\begin{corollary}
	\label{cor:special-values}Let $M$ be a closed $3$-manifold. Then there exists an integer $n\geq 1$ and a polynomial $R\in \Z[A^{\pm 1}][X_1,\ldots, X_n],$ such that, for any $\zeta\in \C$ which is not a root of a non-zero polynomial of the form $R(A^{k_1},\ldots,A^{k_n})$ with $k_i \in \Z,$ we have
	$$\dim_{\C}Sk_{\zeta}(M) <+\infty.$$
\end{corollary}
Corollary \ref{cor:special-values} implies in particular that there are algebraic numbers $\zeta$ such that $\dim_{\C}Sk_{\zeta}(M) <+\infty,$ see Lemma \ref{lemma:special-roots}.

Finally, we have the following important corollary:
\begin{corollary}\label{cor:peripheral-ideal}
	Let $M$ be a compact oriented $3$-manifold with boundary not a disjoint union of spheres, and let $x\in Sk(M,\Z[A^{\pm 1}]).$ Then there exists $z\in Sk(\partial M,\Z[A^{\pm 1}])$ such that $z\neq 0$ and $z\cdot x=0.$
\end{corollary}
When taking $M=S^3\setminus L$ a link complement in $S^3,$ and $x=\emptyset \in Sk(S^3\setminus L,\Z[A^{\pm 1}])),$ we get that the peripheral ideal (defined in \cite{FGL02} by Frohman, Gelca and Lofaro) of any link is non-empty, thus answering a long-standing open question. In particular, by \cite{FGL02}, this gives a geometric proof that the colored Jones polynomial function of a knot is $q$-holonomic, a fact first proved by Garoufalidis and L\^e in \cite{GL05} by algebraic methods.

\textbf{Acknowledgements:} Both authors were partially supported by the ANR project "NAQI-34T" (ANR-23-ERCS-0008) over the course of this work. The first author was also partially supported by the FNRS in his ”Research Fellow” position at UCLouvain.
 The second author thanks David Jordan, Julien March\'e, Gregor Masbaum and Ramanujan Santharoubane for helpful discussions.

\section{Conventions and definitions}
\subsection{Skein modules}

In this subsection we give the basic definition of the Kauffman bracket skein module and we fix the notation we are going to use throughout the paper.

\begin{definition}\label{dfn:KBSMZ} For $M$ a compact oriented $3$-manifold, the Kauffman bracket skein module of $M$ with $\Z[A^{\pm1}]$ coefficients, denoted by $Sk(M,\Z[A^{\pm1}])$, is the quotient of the free $\Z[A^{\pm 1}]$-module generated by isotopy classes of framed links in $M,$ by the Kauffman relations, which are the following relations between framed links that are identical in the complement of a ball:
	\begin{center}
		\def \svgwidth{1.1\columnwidth}
		%% Creator: Inkscape 1.0 (4035a4fb49, 2020-05-01), www.inkscape.org
%% PDF/EPS/PS + LaTeX output extension by Johan Engelen, 2010
%% Accompanies image file 'kauffman.pdf' (pdf, eps, ps)
%%
%% To include the image in your LaTeX document, write
%%   \input{<filename>.pdf_tex}
%%  instead of
%%   \includegraphics{<filename>.pdf}
%% To scale the image, write
%%   \def\svgwidth{<desired width>}
%%   \input{<filename>.pdf_tex}
%%  instead of
%%   \includegraphics[width=<desired width>]{<filename>.pdf}
%%
%% Images with a different path to the parent latex file can
%% be accessed with the `import' package (which may need to be
%% installed) using
%%   \usepackage{import}
%% in the preamble, and then including the image with
%%   \import{<path to file>}{<filename>.pdf_tex}
%% Alternatively, one can specify
%%   \graphicspath{{<path to file>/}}
%% 
%% For more information, please see info/svg-inkscape on CTAN:
%%   http://tug.ctan.org/tex-archive/info/svg-inkscape
%%
\begingroup%
  \makeatletter%
  \providecommand\color[2][]{%
    \errmessage{(Inkscape) Color is used for the text in Inkscape, but the package 'color.sty' is not loaded}%
    \renewcommand\color[2][]{}%
  }%
  \providecommand\transparent[1]{%
    \errmessage{(Inkscape) Transparency is used (non-zero) for the text in Inkscape, but the package 'transparent.sty' is not loaded}%
    \renewcommand\transparent[1]{}%
  }%
  \providecommand\rotatebox[2]{#2}%
  \newcommand*\fsize{\dimexpr\f@size pt\relax}%
  \newcommand*\lineheight[1]{\fontsize{\fsize}{#1\fsize}\selectfont}%
  \ifx\svgwidth\undefined%
    \setlength{\unitlength}{310.27629089bp}%
    \ifx\svgscale\undefined%
      \relax%
    \else%
      \setlength{\unitlength}{\unitlength * \real{\svgscale}}%
    \fi%
  \else%
    \setlength{\unitlength}{\svgwidth}%
  \fi%
  \global\let\svgwidth\undefined%
  \global\let\svgscale\undefined%
  \makeatother%
  \begin{picture}(1,0.08365284)%
    \lineheight{1}%
    \setlength\tabcolsep{0pt}%
    \put(0,0){\includegraphics[width=\unitlength,page=1]{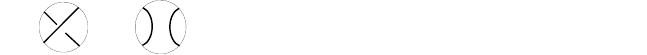}}%
    \put(0.1543921,0.02370045){\color[rgb]{0,0,0}\makebox(0,0)[lt]{\lineheight{0}\smash{\begin{tabular}[t]{l}$=A$\end{tabular}}}}%
    \put(0.29591666,0.02367871){\color[rgb]{0,0,0}\makebox(0,0)[lt]{\lineheight{0}\smash{\begin{tabular}[t]{l}$+A^{-1}$\end{tabular}}}}%
    \put(0.57577529,-0.02838106){\color[rgb]{0,0,0}\makebox(0,0)[lt]{\lineheight{0}\smash{\begin{tabular}[t]{l} \end{tabular}}}}%
    \put(0.53325876,0.02471831){\color[rgb]{0,0,0}\makebox(0,0)[lt]{\lineheight{0}\smash{\begin{tabular}[t]{l}$L \ \coprod$ \end{tabular}}}}%
    \put(0,0){\includegraphics[width=\unitlength,page=2]{kauffman.pdf}}%
    \put(0.67966381,0.02531714){\color[rgb]{0,0,0}\makebox(0,0)[lt]{\lineheight{0}\smash{\begin{tabular}[t]{l}$=(-A^2-A^{-2}) L$\end{tabular}}}}%
    \put(-0.00304793,0.0192645){\color[rgb]{0,0,0}\makebox(0,0)[lt]{\lineheight{0}\smash{\begin{tabular}[t]{l}K1:\end{tabular}}}}%
    \put(0.47835686,0.02531722){\color[rgb]{0,0,0}\makebox(0,0)[lt]{\lineheight{0}\smash{\begin{tabular}[t]{l}K2:\end{tabular}}}}%
    \put(0,0){\includegraphics[width=\unitlength,page=3]{kauffman.pdf}}%
  \end{picture}%
\endgroup%

	\end{center}
\end{definition}
\begin{definition}\label{dfn:KBSM}
	For $M$ as above, the Kauffman bracket skein module with $\Q(A)$ coefficients, simply denoted with $Sk(M)$, is defined as $Sk(M,\Z[A^{\pm1}])\otimes \Q(A)$. Throughout the rest of the paper we will simply call this object the skein module of $M$.
	
	Furthermore in some applications we will look at $Sk(M,\Z[A^{\pm1}])\otimes_{A=\zeta}\C$ for $\zeta\in\C^*$; this indicates the tensor product of $\Z[A^{\pm1}]$-modules where $A$ acts on $\C$ as multiplication by $\zeta$. This object will be denoted $Sk_\zeta(M)$, and will be called the skein module of $M$ at $\zeta$.
\end{definition}

\begin{remark}
	It is well known (see for example \cite[Proposition 2.2]{Prz99}) that removing a ball from a compact $3$-manifold does not change the skein module (in other words, the natural inclusion between the two manifolds induces an isomorphism). Notice furthermore that in the statement of Theorem \ref{thm:main}, if a component of $\partial M$ is a sphere, then it contains no curves of the pants decomposition of $\partial M$. This means that throughout the paper we can make the tacit assumption that no component of $\partial M$ is a sphere.
\end{remark}

\subsection{Graphs}

In what follows we will consider  graphs that are "lollipop trees", i.e. they are obtained from a tree by gluing \emph{loop edges} (i.e., edges whose $2$ endpoints coincide) onto the leaves. All these graphs will only have trivalent vertices. We will also look at paths in these lollipop trees; in this context a path $\gamma$ is simply a connected union of edges such that at each vertex of the graph, $\gamma$ contains at most two distinct edges containing the vertex (so for example, the subset of Figure \ref{fig:loopedpath} is a path). For such a path $\gamma$, we denote with $\partial\gamma$ the loop edges of $\gamma$, if any, and with $\mathring{\gamma}$ the path obtained from $\gamma$ by removing $\partial\gamma$. Similarly we denote with $\mathring{\Gamma}$ the subgraph of $\Gamma$ obtained by removing all loop edges.

\begin{figure}
	\includegraphics[scale=0.4]{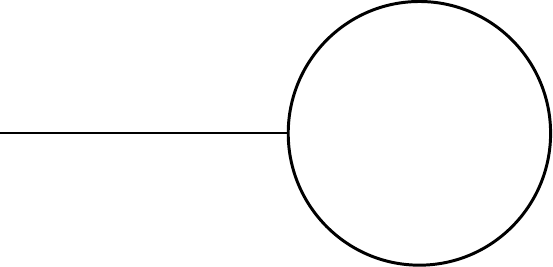}
	\caption{A path in $\Gamma$}
	\label{fig:loopedpath}
\end{figure}

In Section \ref{sec:polytope-boundary}, we will need to consider graphs where each half-edge has an orientation; these are called \emph{bidirected graphs}, and were introduced in \cite{EJ70}. We will only use the terminology related to these objects, rather than any specific result.

\begin{figure}
	\centering
	\begin{minipage}{.45\textwidth}
		\centering   
		
		\includegraphics[scale=0.35]{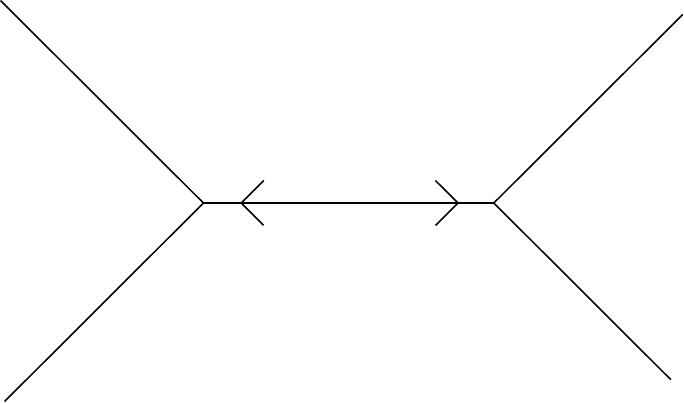}
		
	\end{minipage}   
	\begin{minipage}{.45\textwidth}
		\centering    
		\includegraphics[scale=0.35]{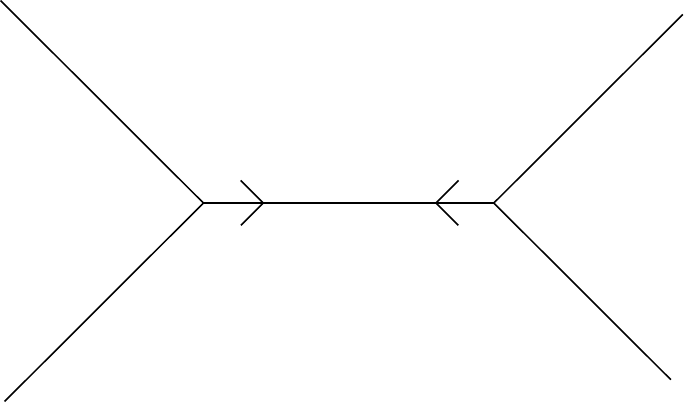}

	\end{minipage}
	\caption{Extraverted (left) and introverted (right) edges.}\label{fig:edges}
\end{figure}

\begin{definition}
	A \emph{half-edge} of a graph $\Gamma$ is a pair $(e,v)$ where $e$ is an edge of $\Gamma$ and $v$ is one of the endpoints of $v$. We denote the set of half-edges of $\Gamma$ by $\mathcal{H}$. An \emph{orientation} for a half-edge $(e,v)$ is the choice of a sign $\pm$; we will depict it as an arrow pointing towards $v$ if the sign is $+$, and away from $v$ if it is $-$. A \emph{bidirected graph} is a graph $\Gamma$ with a choice of sign for each half-edge, i.e. a map $o:\mathcal{H}\ra \{\pm\}$. Notice that if $e$ is an edge loop then it will only have one half-edge; for simplicity of notation we will behave as if it has two coinciding half-edges.
	
	We say that an edge in a bidirected graph is \emph{directed} if its half edges  have different signs, \emph{introverted} if the signs are both $-$ and \emph{extraverted} if they are both $+$ (see Figure \ref{fig:edges}). A directed edge has a natural orientation (in the usual sense of directed graphs). A \emph{path} is a connected union of half edges, such that at each vertex either none or $2$ adjacent half-edges are included in the path. A \emph{directed path} is a path in $\Gamma$ such that if it includes both half-edges of an edge, the edge must be directed, and at each of its vertices the orientations of the two half-edges are opposite. We will use these paths to define operators based on the edges they contain; we use the convention that a path contain an edge if it contains at least one of its half-edges.
	
	We will also allow bidirected graphs to have some edges where only one half-edge has a direction; if the sign of the half-edge is $+$ we say that the edge is a \emph{root}, we say it is a \emph{leaf} otherwise.
	
\end{definition}

\section{The localized skein module of compression bodies}

\label{sec:filtration}

The goal of this section will be to prove the main topological lemma that will be used to produce recurrence relations among the generators of $Sk^{\overline{\lambda}}(M).$ To do so, we will investigate the properties of the localized skein modules of compression bodies.

A \emph{compression body} is a $3$-manifold with boundary obtained by the following process. Start with $\Sigma$ be a connected closed compact oriented surface. Pick a family $\mathcal{C}$ of disjoint, non-trivial, pairwise non-parallel simple closed curve on $\Sigma.$ (note that $\mathcal{C}$ is not assumed to be maximal here). Attaching $2$-handles onto $\Sigma\times [0,1]$ along curves in $\mathcal{C}\times \lbrace 1 \rbrace,$ we obtain a compression body $C.$ Note that the boundary of $C$ consists of $\Sigma\times \lbrace 0 \rbrace$ and another surface $\Sigma',$ which is obtained from $\Sigma \times \lbrace 1 \rbrace$ by compressing along the curves of $\mathcal{C}.$ We will call $\Sigma \times \lbrace 0 \rbrace$ (resp. $\Sigma'$) the negative (resp. the positive) boundary of the compression body $C.$

Finally, if $\Sigma'$ contains components that are spheres $S^2,$ we allow to fill them up with $3$-balls and will still call the resulting manifold a compression body.

In particular, if $\mathcal{C}$ is a pants decomposition of $\Sigma,$ and we fill all spheres in $\partial C$ with $3$-balls, we obtain a handlebody $H=C$ with boundary $\Sigma.$

Note that in general, the positive boundary of $C$ may be disconnected. 
The following theorem is a basic result of $3$-dimensional topology, which can proved by Morse theory, see \cite{Mil:book}

\begin{theorem}
	Let $M$ be a connected compact oriented $3$-manifold. Then $M$ admits a splitting
	$$M=\overline{H}\underset{\Sigma}{\cup}C$$
	where $\Sigma$ is a closed compact connected oriented surface embedded in $M,$ where $H$ is a handlebody with boundary $\Sigma,$ and where $C$ is a compression body with negative boundary $\Sigma$ and positive boundary $\partial M.$
\end{theorem}

In the following, we will be interested in the localized skein module $Sk^{\overline{\lambda}}(C)$ for a fixed pants decomposition $\overline{\lambda}$ of $\partial M,$ which is also the positive boundary of $C.$ The following lemma will allow us to pick a pants decomposition of the negative boundary of $C$ that is in some way compatible with $\overline{\lambda}.$
\begin{lemma}\label{lemma:compression-pantsdec}
	Let $C$ be a compression body with negative boundary $\Sigma$ and positive boundary $S,$ with $S$ containing no spheres, and let $\overline{\lambda}$ be a pants decomposition of $S.$ Then, there exists a pair of pants decomposition $\mathcal{C}$ of $\Sigma,$ such that each curve $\gamma$ of $\mathcal{C}$ either bounds a disk in $C$ or cobounds an annulus with a curve in $S$ isotopic to a curve of $\overline{\lambda}.$ Moreover, the collection of such disks and annuli can be chosen disjoint.
\end{lemma}
\begin{proof}
	First notice that the Lemma is immediate if $C$ is either a handlebody or a thickened surface; in the first case we choose any disk system for the handlebody, and in the second we can choose the same collection of curves.
	
	Now consider two compression bodies $C_1,C_2$ and $\overline{\lambda_1},\overline{\lambda_2}$ two pants decompositions of the positive boundaries of $C_1,C_2$ respectively. Suppose $\widetilde{\lambda_1},\widetilde{\lambda_2}$ are pants decompositions of the negative boundaries of $C_1,C_2$ respectively satisfying the Lemma. Then consider the boundary connected sum $C_1\natural C_2$ along two disks in the negative boundary of $C_1$ and $C_2.$ Its positive boundary is the disjoint union of the positive boundaries of $C_1,C_2$, and $\overline{\lambda_1}\cup\overline{\lambda_2}$ is a pants decomposition for it. Then $\widetilde{\lambda_1}\cup\widetilde{\lambda_2}\cup \lambda$, where $\lambda$ is the boundary of the gluing disk, is a pants decomposition of the negative boundary satisfying the Lemma. 
	
	Since any compression body $C$ is a boundary connected sum of thickened surfaces and a handlebody, we can build a pants decomposition of the negative boundary by doing so in each piece of the boundary connected sum decomposition, and recombining them as above.
\end{proof}
\color{black}
Now, we fix a compression body $C,$ with negative boundary $\Sigma$ and positive boundary $S,$ a pants decomposition $\overline{\lambda}$ of $S$ and a pants decomposition of  $\Sigma$ given by Lemma \ref{lemma:compression-pantsdec}, denoted by $\mathcal{C}=\lbrace \alpha_1,\ldots, \alpha_{3g-3}\rbrace$ . Let us note that $Sk(C)$ is generated by links in $\Sigma \times [0,1],$ since $C$ is obtained from $\Sigma\times [0,1]$ by $2$-handle attachments. The same is true for the localized skein module $Sk^{\overline{\lambda}}(C)$ and therefore the latter is spanned by multicurves in $\Sigma.$

We will denote by 
$$\pi:Sk(\Sigma)\longrightarrow Sk^{\overline{\lambda}}(C)$$
the map which is the composition of the natural surjective map $Sk(\Sigma)\longrightarrow Sk(C)$ coming from the embedding $\Sigma\times [0,1] \subset C$ and the localization map $Sk(C)\longrightarrow Sk^{\overline{\lambda}}(C).$

For two multicurves $\gamma,\delta$ in $\Sigma,$ we will denote by $I(\gamma,\delta)$ the \emph{geometric intersection number} of $\gamma$ and $\delta,$ which is the minimal number of intersections among all representatives of $\gamma$ and $\delta.$

For $n\geq 1,$ we introduce the following subspace of $Sk^{\overline{\lambda}}(C):$

$$\mathcal{F}_n=\mathrm{Span}_{\mathbb{Q}(A,\overline{\lambda})}\left(\pi(\gamma),\ | \ \gamma \subset \Sigma \ \textrm{multicurve such that} \ \forall \alpha_i \in \mathcal{C}, I(\gamma,\alpha_i)\leq n \right).$$

\begin{lemma}\label{lem:dimPoly} We have $\dim_{\Q(A,\overline{\lambda})}(\mathcal{F}_n)=O(n^{3g-3}).$
\end{lemma}
The proof of this lemma will make use of Dehn-Thurston coordinates on the set of multicurves on $\Sigma,$ which we recall below.

Given $\Sigma$ a compact closed oriented surface of genus $g$, and $\mathcal{C}=\lbrace c_1,\ldots,c_{3g-3} \rbrace,$ a pants decomposition of $\Sigma,$  the set of isotopy classes of multicurves $\gamma$ on $\Sigma$ may be parametrized by the so-called Dehn-Thurston coordinates $(I_j(\gamma),tw_j(\gamma)).$

 The intersection coordinates $I_j(\gamma)\in \N$ are simply $I(\gamma,c_j),$ where $I$ is the geometric intersection number. Those coordinates satisfy some parity conditions: if $c_i,c_j,c_k$ cobound a pair of pants, then $I_i(\gamma)+I_j(\gamma)+I_k(\gamma)$ is even. 
 
 The twist coordinates $tw_j(\gamma)$ may be most conveniently defined as relative coordinates. If $I_j(\gamma)=0,$ then $tw_j(\gamma)\in \N$ is the number of curves in $\gamma$ that are isotopic to $c_i.$ If however $I_j(\gamma)=I_j(\gamma_+)=k>0,$ then we will set $tw_i(\gamma_+)-tw_i(\gamma)=1$ if $\gamma_+$ is obtained from $\gamma$ by $1/k$-fractional Dehn twist along $c_i$ (see Figure \ref{fig:twistFrac}).

  \begin{figure}[h]
 	\centering
 	\def \svgwidth{.5\columnwidth}
 	%% Creator: Inkscape 1.3 (0e150ed6c4, 2023-07-21), www.inkscape.org
%% PDF/EPS/PS + LaTeX output extension by Johan Engelen, 2010
%% Accompanies image file 'twistFrac.pdf' (pdf, eps, ps)
%%
%% To include the image in your LaTeX document, write
%%   \input{<filename>.pdf_tex}
%%  instead of
%%   \includegraphics{<filename>.pdf}
%% To scale the image, write
%%   \def\svgwidth{<desired width>}
%%   \input{<filename>.pdf_tex}
%%  instead of
%%   \includegraphics[width=<desired width>]{<filename>.pdf}
%%
%% Images with a different path to the parent latex file can
%% be accessed with the `import' package (which may need to be
%% installed) using
%%   \usepackage{import}
%% in the preamble, and then including the image with
%%   \import{<path to file>}{<filename>.pdf_tex}
%% Alternatively, one can specify
%%   \graphicspath{{<path to file>/}}
%% 
%% For more information, please see info/svg-inkscape on CTAN:
%%   http://tug.ctan.org/tex-archive/info/svg-inkscape
%%
\begingroup%
  \makeatletter%
  \providecommand\color[2][]{%
    \errmessage{(Inkscape) Color is used for the text in Inkscape, but the package 'color.sty' is not loaded}%
    \renewcommand\color[2][]{}%
  }%
  \providecommand\transparent[1]{%
    \errmessage{(Inkscape) Transparency is used (non-zero) for the text in Inkscape, but the package 'transparent.sty' is not loaded}%
    \renewcommand\transparent[1]{}%
  }%
  \providecommand\rotatebox[2]{#2}%
  \newcommand*\fsize{\dimexpr\f@size pt\relax}%
  \newcommand*\lineheight[1]{\fontsize{\fsize}{#1\fsize}\selectfont}%
  \ifx\svgwidth\undefined%
    \setlength{\unitlength}{346.97938453bp}%
    \ifx\svgscale\undefined%
      \relax%
    \else%
      \setlength{\unitlength}{\unitlength * \real{\svgscale}}%
    \fi%
  \else%
    \setlength{\unitlength}{\svgwidth}%
  \fi%
  \global\let\svgwidth\undefined%
  \global\let\svgscale\undefined%
  \makeatother%
  \begin{picture}(1,0.61026734)%
    \lineheight{1}%
    \setlength\tabcolsep{0pt}%
    \put(0,0){\includegraphics[width=\unitlength,page=1]{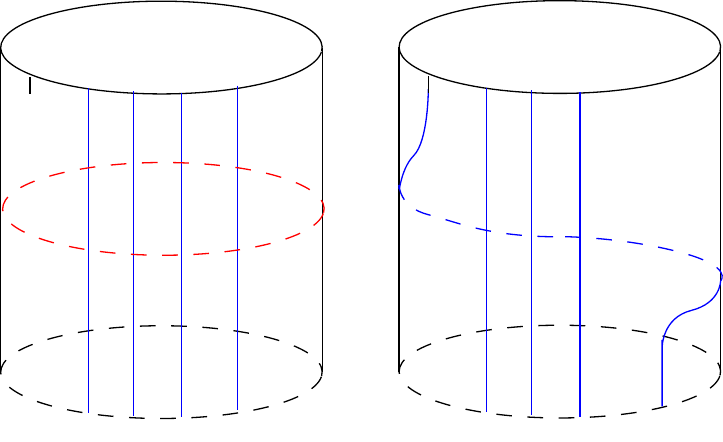}}%
    \put(0.19074587,0.00175632){\color[rgb]{0,0,1}\makebox(0,0)[lt]{\lineheight{1.25}\smash{\begin{tabular}[t]{l}$\gamma$\end{tabular}}}}%
    \put(0.46810488,0.32799766){\color[rgb]{1,0,0}\makebox(0,0)[lt]{\lineheight{1.25}\smash{\begin{tabular}[t]{l}$c$\end{tabular}}}}%
    \put(0.73454438,0.006554){\color[rgb]{0,0,1}\makebox(0,0)[lt]{\lineheight{1.25}\smash{\begin{tabular}[t]{l}$\gamma_+$\end{tabular}}}}%
  \end{picture}%
\endgroup%

 	\caption{The curve $\gamma_+$ is obtained from the curve $\gamma$ by fractional $1/4$-Dehn twist along the curve $c.$}
 	\label{fig:twistFrac}
 \end{figure}
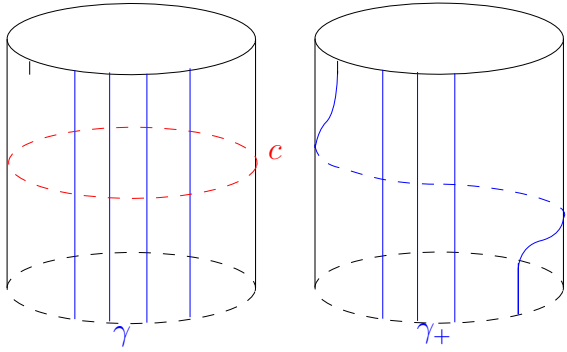
 
 The coordinates $tw_i(\gamma)$ are then valued in $\Z,$ and one can check isotopy classes of multicurves on $\Sigma$ are in bijective correspondance with the Dehn-Thurston coordinates $(I_i(\gamma),tw_i(\gamma))_{1\leq i \leq 3g-3}$ satisfying the above mentioned conditions.
 
 We remark that one may fix an origin to twist coordinates by picking a trivalent ribbon graph $\Gamma$ embedded on $\Sigma$ that is dual to the pants decomposition $\mathcal{C}$ (meaning, each pants contains a single trivalent vertex of $\Gamma$ and each curve $c_i$ intersects a single edge of $\Gamma$ along an interval), and requiring for that multicurves supported on $\Gamma,$ the twist coordinates vanish. 
 
 We will however not need this and will work with twist coordinates viewed as relative coordinates.
 
\begin{proof}[Proof of Lemma \ref{lem:dimPoly}]
	We will show that $\mathcal{F}_n$ is spanned over $\Q(A,\overline{\lambda})$ by the $\pi(\gamma)$ for $\gamma$ multicurves that have all their twist coordinates in $\lbrace 0,1\rbrace.$ Because of the above description of the Dehn-Thurston coordinates, this will imply that $\dim_{\Q(A,\overline{\lambda})}(\mathcal{F}_n)\leq (2n)^{3g-3}$ and prove the lemma.
		
	We will prove this above claim more generally for the spaces 
	$$\mathcal{F}_{n_1,\ldots,n_{3g-3}}=\mathrm{Span}_{\mathbb{Q}(A,\overline{\lambda})}\left(\pi(\gamma),\ | \ \gamma \subset \Sigma \ \textrm{multicurve s.t.} \ \forall \alpha_i \in \mathcal{C}, I(\gamma,\alpha_i)\leq n_i \right),$$
	where $(n_1,\ldots,n_{3g-3})\in \N^{3g-3},$ by induction on $N=n_1+\ldots+n_{3g-3}.$ 
	
	If $N=0$ then the multicurves with intersection coordinates all zero are unions of parallel copies of each of the curves $c_i.$ But the curves $c_i$ either bound disks, in which case we can erase them up to multiplying by $-A^2-A^{-2},$ or are isotopic to curves $\lambda_{j_i}$ on the positive boundary of $C.$ In the end, any such multicurve is equal to $P(A,\lambda_1,\ldots,\lambda_{3g-3})\emptyset,$ for some polynomial $P.$ This proves the case $N=0.$
	
	Now, consider $(n_1,\ldots,n_{3g-3})\in \N^{3g-3}$ with $n_1+\ldots+n_{3g-3}=N$ and assume that the claim is proved for $n_1+\ldots+n_{3g-3}<N.$
	
	If some $n_i=0$, then a multicurve $\gamma$ with intersection coordinates $(n_1,\ldots,n_{3g-3})$ is proportional to one with $tw_i(\gamma)=0$ by exactly the same argument as above. Now let us consider $i$ such that $n_i>0.$ For $k\in \Z,$ let $\gamma_k$ be obtained from $\gamma$ by $k/n_i$-fractional Dehn twist along $c_i.$ There are two cases again:
	
	Either $c_i$ bounds a disk in $C.$ In that case, one may isotope $\gamma_k$ to a framed link in $\Sigma\times [0,1],$ which is obtained from $\gamma$ by modifying $\gamma$ in a neighborhood of $c_i,$ replacing $n_i$ parallel strands by a braid. Then by the Kauffman relations, one gets for any $k\in \Z,$
	$$\pi(\gamma_{k+1})=A^{6+(n_i-1)}\pi(\gamma_k) + \textrm{lower order terms},$$
	where by lower order terms we mean an element of $\mathcal{F}_{(n_1,\ldots,n_i-1,\ldots,n_{3g-3})}.$
	If on the other hand, $c_i$ is isotopic to a curve $\lambda_j,$ then we get
	$$\lambda_j \cdot \gamma_k= A^{-n_i}\gamma_{k+1} +A^{n_i}\gamma_{k-1} + \textrm{lower order terms}.$$
	This shows that one may span $\mathcal{F}_{(n_1,\ldots,n_{3g-3})}$ by multicurves with all twist coordinates in $\lbrace 0,1 \rbrace,$ up to lower order terms, and we conclude by the induction hypothesis.

\end{proof}

The following lemma will be the key topological ingredient of our proof. We will use that to get various elements in $Sk(\Sigma)$ that vanish in $Sk^{\overline{\lambda}}(C).$ When we view these elements as curve operators acting on $Sk(H),$ we will get recurrence relations for the generators of $Sk^{\overline{\lambda}}(M),$ as we will explain in the next section.
\begin{lemma}\label{lem:linDep} Let $c_1,\ldots c_{3g-2}$ be elements in $Sk(\Sigma).$ Then for $n$ large enough, the elements in $\lbrace \pi(c_1^{n_1}\ldots c_{3g-2}^{n_{3g-2}}) \ | \ n_i\leq n \rbrace\subseteq Sk(C)$ are linearly dependent over $\Q(A,\overline{\lambda}).$
\end{lemma}
\begin{proof}
	For a general skein element $c_i\in S(\Sigma)$, written as a  linear combination of multicurves $\gamma_1,\dots,\gamma_k$ with non-zero coefficients, we define the intersection $I(c_i,\alpha_j)$ as the maximum of $I(\gamma_i,\alpha_j)$ among all $\gamma_i$'s.

Now let $B=\max\left(\underset{i=1}{\overset{3g-2}{\sum}}I(c_i,\alpha_j), \ 1\leq j \leq 3g-3 \right).$ We claim that all the elements in $\lbrace \pi(c_1^{n_1}\ldots c_{3g-2}^{n_{3g-2}}) \ | \ n_i\leq  n \rbrace$ are in $\mathcal{F}_{nB}.$ Indeed, one can express them as a linear combination of multicurves in $\Sigma$ by writing each $c_i$ as a linear combination of multicurves, taking the union of the $c_i^{n_i}$ and resolving all crossings. But the union of the curves making up each $c_i^{n_i}$ intersects each $\alpha_j$ at most $nB$ times, so after resolving the crossings we obtain a sum of multicurves in $\mathcal{F}_{nB}.$ 
	
	Now we remark that $\mathrm{dim}(\pi(\mathcal{F}_{nB}))=O(n^{3g-3})$ by Lemma \ref{lem:dimPoly}, so the $(n+1)^{3g-2}$ elements in  $\lbrace \pi(c_1^{n_1}\ldots c_{3g-2}^{n_{3g-2}}) \ | \ n_i\leq n \rbrace$ must be linearly dependent for $n$ large enough.
\end{proof}
\iffalse
We can tweak the proof of the previous lemma to get a slightly sharper result.

\begin{lemma}\label{lem:linDepcond} Let $c_1,\ldots c_{3g-2}$ be elements in $Sk(\Sigma),$ and let $\mathbf{C}\subseteq \R_{\geq 0}^{3g-2}$ be a cone with non-empty interior. Then for $n$ large enough, the elements in 
	$$\lbrace \pi(c_1^{n_1}\ldots c_{3g-2}^{n_{3g-2}}) \ | \ n_i\leq n \textrm{ and } (n_1,\dots,n_{3g-2})\in \mathbf{C}\rbrace$$ are linearly dependent over $\Q(A,\overline{\lambda}).$
\end{lemma}

\begin{proof}
	The proof is entirely similar to that of Lemma \ref{lem:linDepcond}, after noticing that the number of points $(n_1,\ldots,n_{3g-2})\in \mathbf{C}$ with $n_i\leq n$ grows like a positive constant times $n^{3g-2},$ since $\mathbf{C}$ has non-empty interior.
\end{proof}
\fi

\section{Operators on the skein module}
\label{sec:curvop}
In this section we study curve operators, which we will use to produce recurrence relations on skein elements; we first start with some explicit formulas and then study their \emph{minimal shifts}, which will be a key concept in the proof of finiteness.
\subsection{Basis of the skein module of handlebodies}
\label{sec:graph}
\begin{definition}\label{def:coloring}
	A \emph{coloring} of a graph $\Gamma$ is a map $d:\mathcal{E}\ra \Z$, where $\mathcal{E}$ is the set of edges of $\Gamma$. Furthermore, if $v$ is a vertex of $\Gamma$ and $e$ is an adjacent edge, we will set $d(v,e):=d(e_1)+d(e_2)-d(e)$, where $e_1$ and $e_2$ are the edges different from $e$ (but possibly equal to each other) adjacent to $v$.
\end{definition}

\begin{definition}\label{def:admcolor}
	A coloring $d$ of a graph $\Gamma$ is admissible if and only if:
	\begin{itemize}
		\item for any edge $e$, $d(e)\geq 0$;
		\item for any directed vertex $(v,e)$, $d(v,e)\geq 0,$ and $d(v,e)$ is even.
	\end{itemize}
	
	We call $\Delta\subseteq \Z^{\mathcal{E}}$ the polytope of admissible colorings of $\Gamma$. Technically, $\Delta$ is the intersection of a polytope in $\R^{\mathcal{E}}$ with a lattice, but by a slight abuse of terminology we will call it a polytope nonetheless.
\end{definition}

For a banded trivalent graph $\Gamma,$ admissible colorings are in bijective correspondance with the set of multicurves on $\Gamma.$ Indeed, for each edge $e\in \mathcal{E},$ let $a_e$ be the co-core of the corresponding edge in $\Gamma.$  For any multicurve $\gamma$ on $\Gamma,$ we have triangular inequality $I(\gamma,a_i)\leq I(\gamma,a_j)+I(\gamma,a_k)$ for any $i,j,k\in \mathcal{E}$ such that the arcs $a_i,a_j,a_k$ cobound an hexagon in $\Gamma.$ Furthermore, the geometric intersections satisfy that $I(\gamma,a_i)+I(\gamma,a_j)+I(\gamma,a_k)\equiv 0 \mod 2$ for any such triple $(i,j,k).$ Conversely, for any map $d:\mathcal{E}\rightarrow \N$ that satisfy such triangular inequalities and parity condition, there is a unique multicurve $\gamma$ on $\Gamma$ such that $I(\gamma,a_i)=d_i,  \ \forall i \in \mathcal{E}.$ Thus for $d\in \Delta,$ we let $\psi_d$ be the corresponding multicurve.

Let $H$ be a handlebody, then one can view $H$ as the thickening of a banded trivalent graph $\Gamma,$ in many ways. The choice of the graph $\Gamma$ in particular gives a choice of pants decomposition $\mathcal{C}$ of $\Sigma=\partial H,$ by considering the meridians of the thickened edges in $\Gamma\times [0,1] \simeq H.$ The graph $\Gamma\times \lbrace 1 \rbrace \subset \Sigma$ is said \emph{dual} to the pants decomposition $\mathcal{C}.$ For any banded trivalent graph $\Gamma_0$ with $3g-3$ edges, and handlebody $H$ of genus $g,$ there is a pants decomposition $\mathcal{C}$ of $\partial H$ and a trivalent graph $\Gamma$ dual to $\mathcal{C}$ which is isomorphic to $\Gamma_0.$ We will use this fact in Section \ref{sec:curvop-nearadm} to work with trivalent graphs that are lollipop trees.

It is well known that multicurves form a basis of skein algebras of surfaces. Therefore, the family $(\psi_d)_{d\in \Delta}$ form a basis of $Sk(H)\simeq Sk(\Gamma).$

However, we will need to consider another basis of $Sk(H),$ which is also indexed by $\Delta,$ and which will be better behaved for the natural action of $Sk(\Sigma)$ on $Sk(H)$ than the basis of multicurves. For $i\geq 0,$ let $f_i$ denote the $i$-th Jones-Wenzl idempotent, and for any $d\in \Delta,$ let $\varphi_d$ be the element of $Sk(\Gamma)$ obtained from the multicurve $\psi_d$ by inserting $f_{d_i}$ on the edge $i\in \mathcal{E}.$ 

The following well known fact is Lemma 2.1 of \cite{DS25}:

\begin{lemma}\label{lemma:basis}
	$\{ \varphi_c \, | \, c : \mathcal{E} \to \Z \, \, \text{admissible} \}$ is a basis of $Sk(\Gamma,\Q(A))\simeq Sk(H,\Q(A))$.
\end{lemma}

\subsection{Curve operators}
\label{sec:curvop-lemmas}

Let $\Sigma$ be the boundary of a handlebody $H\simeq \Gamma \times [0,1].$ The skein module $Sk(H)$ has a natural module structure over $Sk(\Sigma),$ obtained by stacking. 

 For $\gamma \in Sk(\Sigma),$ we will call $T^{\gamma}$ the endomorphism of $Sk(H)$ induced by stacking $\gamma.$ The operator $T^{\gamma}$ is called the \emph{curve operator} associated to the skein element $\gamma \in Sk(\Sigma).$ 
 
 We will need to express the action of some special simple closed curves on the basis $\varphi_c.$ For any $e\in \mathcal{E},$ let $\alpha_e$ be the simple closed curve that encircles the edge $e$ of $\Gamma,$ and let $\beta_e$ be the curve represented on Figure \ref{fig:curves}.
  
 \begin{figure}[h]
 	\centering
 	\def \svgwidth{.9\columnwidth}
 	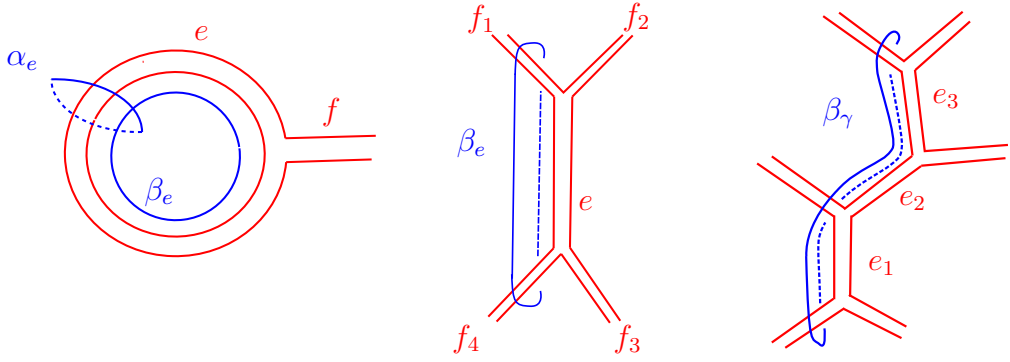
 	\caption{Left: the curves $\alpha_e$ and $\beta_e$ associated to a loop edge $e$. Middle: the curve $\beta_e$ associated to a non-loop edge. Right: the curve $\beta_{\gamma}$ associated to an arc $\gamma$ composed of the $3$ non-loop edges $e_1,e_2,e_3.$ The surface $\Sigma$ is not represented but corresponds to the boundary of a regular neighborhood of the graph $\Gamma,$ represented in red.}
 	\label{fig:curves}
 \end{figure}
 
 For $e\in \mathcal{E},$ we let $\delta_e$ be the Kronecker symbol, e.g. the map such that $\delta_e(f)=0$ if $e\neq f$ and $\delta_e(e)=1.$
\begin{lemma}\label{lemma:curveActions}  Let $c:\mathcal{E} \longmapsto \Z$ be admissible, then we have
	$$T^{\alpha_e}\varphi_c=(-A^{2+2c_e}-A^{-2-2c_e})\varphi_c$$
	If $e$ is a loop edge, and $f$ is the unique edge adjacent to $e,$ then:
	$$T^{\beta_e}\varphi_c=\varphi_{c+\delta_e}+F_{-1}(A,,A^{c_e},A^{c_f})\varphi_{c-\delta_e}$$
	 where $F_{-1}\in \Q(A,X_1,X_2).$ Moreover $F_{-1}(A,A^{c_e},A^{c_f})\neq 0$ if $c-\delta_e$ is admissible.

	If however $e$ is not a loop edge, then:
	\begin{multline*}T^{\beta_e}\varphi_c=R_2(A,A^{c_e},A^{c_{f_1}},A^{c_{f_2}},A^{c_{f_3}},A^{c_{f_4}})\varphi_{c+2\delta_e}
		\\ + R_0(A,A^{c_e},A^{c_{f_1}},A^{c_{f_2}},A^{c_{f_3}},A^{c_{f_4}})\varphi_c + R_{-2}(A,A^{c_e},A^{c_{f_1}},A^{c_{f_2}},A^{c_{f_3}},A^{c_{f_4}})\varphi_{c-2\delta_e}
		\end{multline*}
	 where $f_1,f_2,f_3,f_4$ are the edges adjacent to $e$ (possibly with repetition) and $R_2$, $R_0$, $R_{-2}$ are in $\Q(A,X_1,\ldots, X_5).$
	 Moreover, $R_2(A,A^{c_e},A^{c_{f_1}},\ldots,A^{c_{f_4}})$ (respectively $R_{-2}(A,A^{c_e},A^{c_{f_1}},\ldots,A^{c_{f_4}})$) is non-zero if $c+2\delta_e$ (respectively $c-2\delta_e$) is also admissible.
\end{lemma}
\begin{remark}
	The above expressions do not have meaning at first glance if they contain a $\varphi_{c'}$ with $c'$ not admissible. However, they have to be interpreted with the convention that $\varphi_c=0$ is $c$ is not admissible. In that case, the coefficient in front of $\varphi_c$ may be chosen arbitrarily.  
\end{remark}

For $\gamma$ an arc in $\Gamma,$ formed by consecutive non-loop edges $e_1,\ldots,e_n \in \mathcal{E},$ we define a curve $\beta_{\gamma}$ in Figure \ref{fig:curves}.
\begin{lemma}\label{lemma:longcurveactions}  Let $c:\mathcal{E} \longmapsto \Z$ be admissible, let $\delta$ be an arc in $\Gamma$ composed of edges $e_1,\dots,e_n$, and let $\beta_{\gamma}$ be the associated curve. Then we have
	$$T^{\beta_\gamma}\varphi_c=\sum F_{c'}\varphi_{c'}$$
	
	where $c'$ runs over the colors $c+\sum_i\epsilon_i\delta_{e_i}$ for all choices of $\epsilon_i\in\{-2,0,2\}$ and $F_{c'}$ is a rational fraction in $A$ and $A^{c(e)}$ for $e\in \mathcal{E}.$ 
	Moreover, for $\epsilon \in \lbrace \pm 2 \rbrace^n,$ one has that $F_{c'}\neq 0$ if $c'$ is admissible.
\end{lemma}

To prove Lemmas \ref{lemma:curveActions} and \ref{lemma:longcurveactions} we will need to explain fusion rules, which permit to compute the action of curve operators on the basis $\varphi_c$ of $Sk(H).$ Given a ball $B^3,$ a uni-trivalent graph $\Gamma$ embedded in $B^3,$ such that the univalent vertices are on the boundary of $B^3,$ and an admissible coloring $c$ of $\Gamma$, one may define an element $(\Gamma,c)$ of the relative skein module $Sk(B^3,\partial \Gamma)$ by the same process as the definition of vectors $\varphi_{c}.$ Fusion rules, which were originally computed in \cite{MV94}, give relations between different colored trivalent graphs. They are represented in Figure \ref{fig:fusion}. The coefficients appearing in those equations are as follows. For $n\in \N,$ we set $\lbrace n \rbrace:=A^{2n}-A^{-2n}.$ We remark that $\lbrace n \rbrace\neq 0$ if $n\neq 0.$ Then

\begin{gather*}
	M_+(n)=1, M_-(n)=-\frac{\lbrace n+1\rbrace }{\lbrace n \rbrace}, \ t(n,1)=A^n, \ t(n,-1)=-A^{-(n+2)}
	\\ T(a,b,c,1,1)=1, \ T(a,b,c,1,-1)=\frac{\lbrace \frac{a-b+c}{2}\rbrace}{\lbrace c \rbrace}, \ T(a,b,c,-1,1)=\frac{\lbrace \frac{c-a+b}{2}\rbrace}{\lbrace b \rbrace},
	\\ T(a,b,c,-1,-1)=-\frac{\lbrace \frac{a+b+c}{2}+1\rbrace \lbrace \frac{b+c-a}{2}\rbrace}{\lbrace b \rbrace \lbrace c \rbrace}.
\end{gather*}
We note that for the twist rule, we did not display the negative half twists. There the coefficients are obtained from $t(n,\pm 1)$ by changing $A$ to $A^{-1}.$

\begin{figure}[h]
	\centering
	\def \svgwidth{.5\columnwidth}
	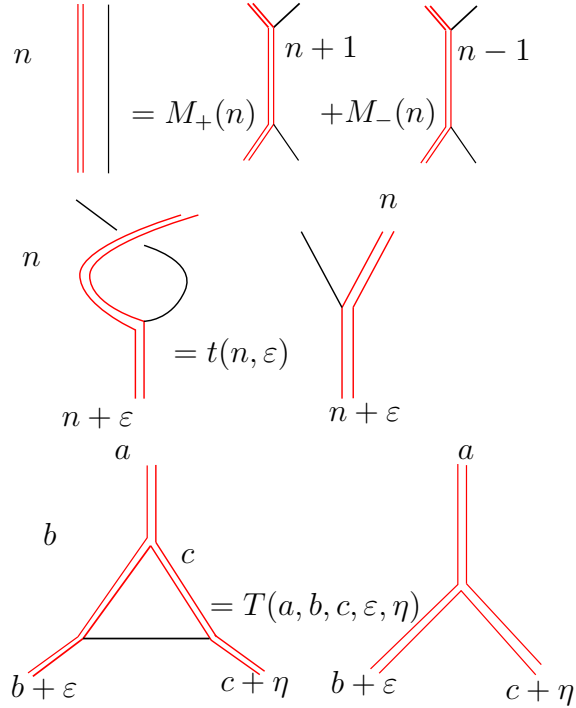
	\caption{The fusion rules. Black edges are colored by $1$ and $\varepsilon,\eta\in \lbrace \pm 1 \rbrace.$ The first, second and third equations will be refered to as the merge rule, the half-twist rule and the triangle rule. All colorings are admissible.}
	\label{fig:fusion}
\end{figure}
We will need the following in order to establish Lemmas \ref{lemma:curveActions} and \ref{lemma:longcurveactions}:
\begin{lemma}
	\label{lemma:come-back}
	We have for $(a,b,c)$ an admissible triple and $\varepsilon,\eta\in \lbrace \pm 1\rbrace:$
	\begin{center}
		\def\svgwidth{0.58\columnwidth}
	%% Creator: Inkscape 1.3 (0e150ed6c4, 2023-07-21), www.inkscape.org
%% PDF/EPS/PS + LaTeX output extension by Johan Engelen, 2010
%% Accompanies image file 'come-back-pattern.pdf' (pdf, eps, ps)
%%
%% To include the image in your LaTeX document, write
%%   \input{<filename>.pdf_tex}
%%  instead of
%%   \includegraphics{<filename>.pdf}
%% To scale the image, write
%%   \def\svgwidth{<desired width>}
%%   \input{<filename>.pdf_tex}
%%  instead of
%%   \includegraphics[width=<desired width>]{<filename>.pdf}
%%
%% Images with a different path to the parent latex file can
%% be accessed with the `import' package (which may need to be
%% installed) using
%%   \usepackage{import}
%% in the preamble, and then including the image with
%%   \import{<path to file>}{<filename>.pdf_tex}
%% Alternatively, one can specify
%%   \graphicspath{{<path to file>/}}
%% 
%% For more information, please see info/svg-inkscape on CTAN:
%%   http://tug.ctan.org/tex-archive/info/svg-inkscape
%%
\begingroup%
  \makeatletter%
  \providecommand\color[2][]{%
    \errmessage{(Inkscape) Color is used for the text in Inkscape, but the package 'color.sty' is not loaded}%
    \renewcommand\color[2][]{}%
  }%
  \providecommand\transparent[1]{%
    \errmessage{(Inkscape) Transparency is used (non-zero) for the text in Inkscape, but the package 'transparent.sty' is not loaded}%
    \renewcommand\transparent[1]{}%
  }%
  \providecommand\rotatebox[2]{#2}%
  \newcommand*\fsize{\dimexpr\f@size pt\relax}%
  \newcommand*\lineheight[1]{\fontsize{\fsize}{#1\fsize}\selectfont}%
  \ifx\svgwidth\undefined%
    \setlength{\unitlength}{236.13410289bp}%
    \ifx\svgscale\undefined%
      \relax%
    \else%
      \setlength{\unitlength}{\unitlength * \real{\svgscale}}%
    \fi%
  \else%
    \setlength{\unitlength}{\svgwidth}%
  \fi%
  \global\let\svgwidth\undefined%
  \global\let\svgscale\undefined%
  \makeatother%
  \begin{picture}(1,0.47158267)%
    \lineheight{1}%
    \setlength\tabcolsep{0pt}%
    \put(0,0){\includegraphics[width=\unitlength,page=1]{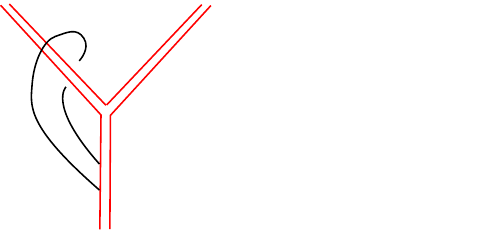}}%
    \put(0.04019037,0.45871262){\color[rgb]{1,0,0}\makebox(0,0)[lt]{\lineheight{1.25}\smash{\begin{tabular}[t]{l}$a$\end{tabular}}}}%
    \put(0.42287982,0.40919799){\color[rgb]{1,0,0}\makebox(0,0)[lt]{\lineheight{1.25}\smash{\begin{tabular}[t]{l}$b$\end{tabular}}}}%
    \put(0.23740966,0.19351565){\color[rgb]{1,0,0}\makebox(0,0)[lt]{\lineheight{1.25}\smash{\begin{tabular}[t]{l}$c$\end{tabular}}}}%
    \put(0.23740966,0.11043183){\color[rgb]{1,0,0}\makebox(0,0)[lt]{\lineheight{1.25}\smash{\begin{tabular}[t]{l}$c+\varepsilon$\end{tabular}}}}%
    \put(0.23489196,0.03154395){\color[rgb]{1,0,0}\makebox(0,0)[lt]{\lineheight{1.25}\smash{\begin{tabular}[t]{l}$c+\varepsilon+\eta$\end{tabular}}}}%
    \put(0.31629746,0.19687256){\color[rgb]{0,0,0}\makebox(0,0)[lt]{\lineheight{1.25}\smash{\begin{tabular}[t]{l}$=F_{\varepsilon,\eta}(A,A^a,A^b,A^c)$\end{tabular}}}}%
    \put(0,0){\includegraphics[width=\unitlength,page=2]{come-back-pattern.pdf}}%
    \put(0.58640922,0.45346207){\color[rgb]{1,0,0}\makebox(0,0)[lt]{\lineheight{1.25}\smash{\begin{tabular}[t]{l}$a$\end{tabular}}}}%
    \put(0.96909858,0.40394745){\color[rgb]{1,0,0}\makebox(0,0)[lt]{\lineheight{1.25}\smash{\begin{tabular}[t]{l}$b$\end{tabular}}}}%
    \put(0.77859319,0.1026634){\color[rgb]{1,0,0}\makebox(0,0)[lt]{\lineheight{1.25}\smash{\begin{tabular}[t]{l}$c+\varepsilon+\eta$\end{tabular}}}}%
  \end{picture}%
\endgroup%

	\end{center}
	where $F_{\varepsilon,\eta}\in \Q(A,X_1,X_2,X_3)$ and $F_{\varepsilon,\varepsilon}(A,A^a,A^b,A^c)\neq 0$ if $(a,b,c+2\varepsilon)$ is admissible.
\end{lemma}
We will call the pattern on the left hand side of Lemma \ref{lemma:come-back} a \emph{come-back pattern}.
\begin{proof}
	This is achieved by a straightforward application of the fusion rules. First we merge the black loop with the edge colored by $a$; thus we can remove the two crossings with two half-twist moves. The resulting graph will have two concentric triangular faces that can be removed with two triangle rules. In formulas, we have that $C(a,b,c,\epsilon,\eta)$ is equal to
	
	$$\left(\sum_{\delta\in\{\pm 1\}} t(a,\delta)^2T(b,a,c,\delta,\epsilon)T(b,a+\delta,c+\epsilon,-\delta,\eta)\right)Y(a,b,c+\epsilon+\eta)$$
	
	where $C(a,b,c,\epsilon,\eta)$ is the come-back pattern and $Y(a,b,c)$ is the vertex colored by $a,b,c+\epsilon+\eta$. The fact that the sum does not cancel out can be either verified directly for each choice of $\epsilon,\eta$ or deduced from the fact that $t(a,1)=A^a,t(a,-1)=A^{-(a+2)}$ and everything else is symmetric under $A\ra A^{-1}$.
\end{proof}
\begin{proof}[Proof of Lemmas \ref{lemma:curveActions} and \ref{lemma:longcurveactions}]
	We note that curves $\beta_e$ where $e$ is a non-loop edge are a special case of curves of the type $\beta_{\gamma},$ so there will be $3$ types of curves to consider to prove both lemmas: curves $\alpha_e,$ curves $\beta_e$ where $e$ is a loop edge, and curves $\beta_{\gamma}$ where $\gamma$ is an arc of consecutive non-loop edges.
	
	For the curves $\alpha_e,$ one has $T^{\alpha_e}\varphi_c=(-A^{2+2c(e)}-A^{-2-2c(e)})\varphi_c,$ by the basic properties of Jones-Wenzl idempotents.
	
	In the case where $e$ is a loop edge, one can compute $T^{\beta_e}\varphi_c$ by first merging $\beta_e$ with $e,$ then simplifying the resulting triangles. Let $f\neq e$ be the edge adjacent to $e.$ One gets that $T^{\beta_e}\varphi_{c}$ is equal to
	\begin{multline*}M_+(c(e))T(c(f),c(e),c(e),1,1)\varphi_{c+\delta_e}+M_-(c(e))T(c(f),c(e),c(e),-1,-1)\varphi_{c-\delta_e}
		\\ =\varphi_{c+\delta_e}+ \frac{ \lbrace c(e)+\frac{c(f)}{2}+1 \rbrace \lbrace c(e)-\frac{c(f)}{2} \rbrace}{\lbrace c(e)\rbrace \lbrace c(e)+1\rbrace}\varphi_{c-\delta_e}.
		\end{multline*}
	It is then clear that the coefficient in $\varphi_{c+\delta_e}$ is non-zero, and the coefficient in $\varphi_{c-\delta_e}$ is too, since if $c-\delta_e$ is admissible then $c(e)-\frac{c(f)}{2}>0.$
	
	Finally, let $e_1,\ldots,e_n$ be a sequence of consecutive non-loop edges, forming an arc $\gamma.$ Let $f_i$ be the third edge adjacent to the trivalent vertex common to $e_i$ and $e_{i+1}.$ Similarly, let $f_0,f_0',f_n,f_n'$ be the edges adjacent to the endpoints of $\gamma,$ with $f_0,f_0'$ adjacent to $e_1$ and $f_n,f_n'$ adjacent to $e_n$ (we don't require them to be different). To compute $T^{\beta_{\gamma}}\varphi_c,$ one uses the merging rule twice for each edge $e_i$, then, possibly after applying some half-twist rules, one is left with getting rid of $2$ triangles supported by edges $i$ and $i+1$ for $1\leq i \leq n-1,$ plus a come-back pattern for each extremity of $\gamma.$ At the end, $T^{\beta_{\gamma}}\varphi_c$ will be a linear combination of the $\varphi_{c'},$ where $c'=c+\sum \varepsilon_i \delta_{e_i}$ for some $\varepsilon_i \in \lbrace 0, \pm 2 \rbrace.$ Indeed, merging one strand colored by $1$ shifts the color by $\pm 1.$	
	Furthermore, it is clear by the fusion rules and by Lemma \ref{lemma:come-back} that the coefficient $F_{c'}$ is a rational fraction in $A$ and the $A^{c(e)}$'s for $e\in \mathcal{E}.$ (in fact, $e=e_i$ for some $i,$ or $e$ adjacent to some $e_i$)
	
	Let us focus on the coefficient in $c'=c+\sum \varepsilon_i \delta_{e_i}$ when $\varepsilon_i\in \lbrace \pm 2 \rbrace$ for all $1\leq i \leq n,$ assuming that $c'$ is admissible. Then one needs to merge twice with the same sign at each edge to get $\varphi_c',$ hence $F_{c'}$ is a sum of one term. More precisely,

	\begin{multline*}F_{c'}=\left(\underset{1\leq i \leq n}{\prod}M_{\varepsilon_i}(c(e_i))\right)		 
		\\ \times \left(\underset{1\leq i \leq n-1}{\prod}T(c(f_i),c(e_i),c(e_{i+1}),\varepsilon_i,\varepsilon_{i+1})T(c(f_i),c(e_i)+\varepsilon_i,c(e_{i+1})+\varepsilon_{i+1},\varepsilon_i,\varepsilon_{i+1})\right)
		\\ \times F_{\varepsilon_1,\varepsilon_1}(A,A^{c(f_0)},A^{c(f_{0}')},A^{c(e_1)})F_{\varepsilon_n,\varepsilon_n}(A,A^{c(f_n)},A^{c(f_{n}')},A^{c(e_n)}),
\end{multline*}
	possibly up to half-twist coefficients, which are of the form $\pm A^k,$ hence non-zero. 
	Note that, for any $\varepsilon \in \lbrace \pm 1 \rbrace,$ the triangle coefficient $T(a,b,c,\varepsilon,\eta)$ is non-zero whenever $(a,b,c)$ and $(a+\varepsilon,b+\eta,c)$ are admissible triples. It is clear when $\varepsilon=\eta=+1.$ When $\varepsilon=\eta=-1$ we have $b+c-a\geq 2$ because $(a,b-1,c-1)$ is admissible, so it is true also. The other two cases are similar. 
	This implies that the triangle coefficients in $F_{c'}$ are all non-zero: indeed if $(c(e_i),c(e_{i+1}),c(f_i))$ and $(c(e_i)+2\varepsilon_i,c(e_{i+1})+2\varepsilon_{i+1},c(f_i))$ are admissible (as implied by the admissibility of $c$ and $c'$) then so is $(c(e_i)+\varepsilon_i,c(e_{i+1})+\varepsilon_{i+1},c(f_i)).$
	 Finally, by Lemma \ref{lemma:come-back}, the factors coming from come-back patterns are also non-zero. 
\end{proof}
\subsection{Minimal shifts}
\label{sec:min-shifts}
The following definitions of shifts and minimal shifts are implicit in Proposition 2.3 and Lemma 2.4 of\cite{DS25}. Recall that a subset of $\Z^ \mathcal{E}$ is \emph{large} if it contains a subset of the form $\cap_{i=1}^{3g-3} \left(c_i+\Delta\right)$, where the $c_i$s are any coloring of $\Gamma$ such that the sum of colors around each vertex is even.
\begin{definition}
	Consider $M$ a monomial in the curve operators. For any $\varphi_c$ basis element of $Sk(H)$, we can apply $M$ to $\varphi_c$; by \cite[Proposition 2.3]{DS25} for $c$ in a large subset, \begin{equation}\label{eqn:operatorformula}M \varphi_c=\sum_{x\in\Z^{\mathcal{E}}}F_x(A^{c(e_1)},\dots,A^{c(e_{3g-3})})\varphi_{c+x},\end{equation} where only a finite number of the $F_x$'s are non-zero. By \cite[Lemma 2.4]{DS25}, the $x$'s for which $F_x\neq 0$ do not depend on the large subset chosen (i.e. it only depends on $M$).  
	
	We say that $x\in \Z^{\mathcal{E}}$ is a \emph{shift for $M$} if $F_x\neq 0$ in the above formula. Fix $k\geq 0,$ an order relation on $\Z^k$ and a group morphism $f:\Z^{\mathcal{E}}\longrightarrow \Z^k.$ We say that $x$ is \emph{a minimal shift for $M$ with respect to $f$} if and only if $f(x)$ is minimal among all shifts of $M$.
	
	We say that an admissible coloring $c$ is \emph{effective} for $M$ if for some minimal shift $x$ of $M$, $c-x$  belongs to a large subset such that \eqref{eqn:operatorformula} holds.	
\end{definition}

Notice that the definition of the minimal shift relies on the choice of $f$; in what follows we will only ever use one choice of $f$ at a time, so we will mostly omit to mention it.

\begin{definition}
	Consider a monomial $M$ in the curve operators, and suppose $x_1,\dots,x_n$ are the minimal shifts of $M$. Then call $\mathcal{P}_M$, the \emph{polytope of minimal shifts of $M$}, the convex hull of $x_1,\dots,x_n$.
\end{definition}

\begin{lemma}\label{lemma:uniqueminimalshift}
	Suppose $M_1$ and $M_2$ are two monomials in the curve operators. Suppose that for any $x$ minimal shift of either of them, $f(x)$ is non-positive. Furthermore suppose that either
	\begin{itemize}
		\item $M_1$ has a unique minimal shift; or
		\item the polytope of minimal shifts of $M_1$ and $M_2$ belong to subspaces that intersect only in $0$.
	\end{itemize}
	Then, the minimal shifts of $M_1M_2$ are exactly $y_1+y_2$, where $y_1$ and $y_2$ are minimal shifts of $M_1$ and $M_2$ respectively.
\end{lemma}
\begin{proof}
	Write $M_2\varphi_c=\sum_x F_x \varphi_{c+x}$ as in Equation \eqref{eqn:operatorformula}, and choose $c$ large enough that for all shifts of $M_2$, Equation \eqref{eqn:operatorformula} also applies to $M_1\varphi_{c+x}$. Then,
	\begin{equation}\label{eqn:productformula}M_1M_2\varphi_c=\sum_{x_1,x_2} F_{x_1}F_{x_2} \varphi_{c+x_1+x_2}\end{equation}
	where the sum is over all shifts $x_1,x_2$ of $M_1,M_2$ respectively. Take $y_1,y_2$ minimal shifts of $M_1,M_2$. First notice that $y_1+y_2$ can appear only once in Equation \eqref{eqn:productformula}; suppose by contradiction that there are $x_1,x_2$ different shifts of $M_1,M_2$ such that $y_1+y_2=x_1+x_2$. This implies that $f(x_1+x_2)=f(y_1+y_2)$ and by minimality and non-positivity, $f(x_1)=f(y_1)$ and $f(x_2)=f(y_2)$ which means that $x_1,x_2$ are minimal shifts for $M_1,M_2$. Then, if $M_1$ has a unique minimal shift, we find that $x_1=y_1$ and thus $x_2=y_2$. If instead the condition about the polytope of minimal shifts holds, we have that $x_1-y_1=y_2-x_2$ and thus must be $0$, obtaining once again that $x_1=y_1$ and $x_2=y_2$. Therefore $y_1+y_2$ appears only once in Equation \eqref{eqn:productformula} with coefficient $F_{y_1}F_{y_2}\neq 0$. It is immediate to see that $y_1+y_2$ must be minimal, and that if either $x_1$ or $x_2$ is not minimal, $x_1+x_2$ cannot be minimal either.
\end{proof}
\begin{corollary}\label{cor:uniqueminimalshift}
	Take some curve operators $T_i$. Suppose that they all have a unique minimal shift $x_i$ with respect to $f$, and suppose furthermore that the $f(x_i)$'s are linearly independent. Then any monomial in the $T_i$'s will also have a unique minimal shift with respect to $f$, equal to the sum (with multiplicity) of the minimal shifts of the $T_i$'s.
\end{corollary}
\begin{proof}
	Suppose $M$ is a monomial in $T_1,\dots,T_n$, and suppose that $T_i$ appears in $M$ with multiplicity $k_i\geq 0$. Then $x=\sum_i k_i x_i$ is a minimal shift for $M$, since the shifts of $M$ are all sums of shifts of the $T_i$'s. Notice that because the $f(x_i)$'s are linearly independent, so are the $x_i$'s, which implies that as we apply each $T_i$, there is only one way to obtain $x$ in the expansion of $M\cdot \varphi_c$. This implies that $F_x$ is simply a product of the $F_{x_i}$'s, and thus is non-zero.
	
	 Suppose that $y$ is also a minimal shift of $M$, i.e. $y$ is a shift of $M$ and $f(y)=f(x)$. Because $y$ must be a sum of shifts of the $T_i$'s, it must by minimality be of the form $\sum_i h_i x_i$. Therefore $f(y)=f(x)$ implies $\sum_i(h_i-k_i)x_i=0$ which means $h_i=k_i$ for all $i$ and $y=x$.
\end{proof}

\begin{lemma}\label{lemma:effectiveness}
	If $c$ is an effective coloring for $M$, and $y$ is a non-negative coloring of $\Gamma$ such that $c+y$ is admissible, then $c+y$ is an effective coloring for $M$.
\end{lemma}
\begin{proof}
	This is an obvious consequence of the definition of effectiveness and of a large subset.
\end{proof}
\begin{corollary}\label{cor:effectiveness}
	If $c$ is an effective coloring for $M_1$ and $M_2$, and the minimal shift $y$ of $M_2$ is non-positive, then $c$ is an effective coloring for $M_1M_2$.
\end{corollary}

\section{Some preliminary lemmas}
\label{sec:prelim-lemmas}
We are almost ready to start the proof of Theorem \ref{thm:main}, but will need a few additional lemmas. The content of this section is elementary and does not use previous sections. The first two lemmas, Lemma \ref{lemma:non-vanishing_poly} and Lemma \ref{lemma:lin_alg} will help us choose the curves of type $\alpha_e$ and $\beta_e$ to which we will apply Lemma \ref{lem:linDep} in the next section. 

\begin{lemma}
	\label{lemma:non-vanishing_poly} Let $D\geq 1,$ and let $V$ be an affine sub-lattice of $\mathbb{Z}^D$ of rank $k.$ Then there is a subset $I=\lbrace i_1<\ldots< i_k \rbrace\subset \lbrace 1,\ldots,D \rbrace$ such that, for any non-zero Laurent polynomial $P\in \mathbb{Q}[A^{\pm 1}][X_1^{\pm 1},\ldots,X_k^{\pm 1}],$ the function 
	$$c\in \mathbb{Z}^D \longrightarrow P(A^{c_{i_1}},\ldots ,A^{c_{i_k}})$$
	is non-vanishing on the complement in $V$ of a finite union of sublattices of rank $<k.$
\end{lemma}
\begin{proof}
	The linear forms $\varphi_i:c\longrightarrow c_i$ for $i\in \lbrace 1,\ldots, D\rbrace$ are linearly independent on $\Q^D,$ hence one can extract a subfamily $(\varphi_i)_{i\in I}$ that is a basis of the space of linear forms on the (rational) direction of $V.$
	
	Now, if we take a non-zero polynomial $P,$ for $P(A^{c_{i_1}},\ldots, A^{c_{i_k}})$ to vanish at some $c\in V,$ one must have that two distinct monomials of $P$ have the same degree in $A.$ If we fix two such monomials in $P,$ this amounts to an equation of the form $a_1\varphi_{i_1}(c)+\ldots + a_k\varphi_{i_k}(c)=d$ where $d$ is a constant and the $a_i$'s do not all vanish. Since the $\varphi_{i_j}$ are linearly independent on the direction of $V,$ this equation is only satisfied on a sublattice of $V$ of rank $k-1.$ Since there are only a finite number of monomials in $P,$ the lemma follows.  
\end{proof}
\begin{lemma}
	\label{lemma:lin_alg}
	Let $v_1,\ldots,v_n$ be a basis of $\Q^n,$ and let $V\subset \Q^n$ a subspace of dimension $d<n.$ Then we can pick $J\subset \lbrace 1,\ldots,n\rbrace$ such that 
	\begin{itemize}
		\item[(1)]$W=\mathrm{Span}_{j\in J}\lbrace v_j \rbrace$ is in direct sum with $V;$ 
		\item[(2)]For any $l\in \lbrace 1,\ldots,n \rbrace \setminus J,$ 
		$$D_l=V\cap \mathrm{Span}_{j\in J\cup \lbrace l \rbrace}\lbrace v_j \rbrace$$
		has dimension $1$;
		\item[(3)]One has $V=\underset{l\in \lbrace 1,\ldots,n\rbrace \setminus J}{\bigoplus}D_l.$
	\end{itemize}
\end{lemma}
\begin{proof}The family $(v_j)_{j\in J}$ may be chosen by the incomplete basis theorem: given $(e_1,\ldots,e_d)$ a basis of $V,$ one can pick vectors among $v_1,\ldots,v_n$ so that the $e_i$'s and the $v_j$'s form a basis of $\Q^n,$ which implies the first claim. 
	Now given such a choice of $J,$ one has that $\dim D_l\geq 1$ since $\dim V+\dim \mathrm{Span}_{j\in J\cup \lbrace l \rbrace}\lbrace v_j \rbrace =n+1,$ and also $\dim D_l\leq 1$ since $\dim V\cap W=0.$ So $\dim D_l=1.$
	
	Finally, 
	$$\underset{l\in \lbrace 1,\ldots,n\rbrace \setminus J }{\bigoplus}D_l=\underset{l\in \lbrace 1,\ldots,n\rbrace \setminus J }{\bigoplus} V\cap  \mathrm{Span}_{j\in J\cup \lbrace l \rbrace}\lbrace v_j \rbrace=V\cap \Q^n=V.$$
\end{proof}

We end this section with the following lemma, which we call the \emph{Chamber lemma}. Loosely speaking, it says that if a function defined on lattice points of a polytope satsfies one-dimensional recurrence relations in every direction, then outside of a union of codimension $\geq 1$ space, it is determined by a finite number of values. This lemma will be of use when in the next section, we find recurrence relations among the generators $\varphi_c$ of $Sk^{\overline{\lambda}}(M).$

\begin{lemma}\label{lemma:chamber}
	(\textit{Chamber lemma}) Let $n,d\geq 1$ be integers, let $e_1,\ldots, e_n$ be a basis of $\Z^n$ and let $V=\underset{i\in I}{\cap} \lbrace \varphi_i(x)\geq a_i \rbrace,$ where $\varphi_i:\Q^n\longrightarrow \Q$ are linear forms and $a_i\in \Q.$ Let $W$ be a $\Z$-module and $f:\Z^n \longrightarrow W$ be a set map satisfying the following condition for some integer $d\geq 1:$
	\begin{itemize}
		\item[(*)] $\forall x \in \Z^n \cap \mathring{V}, \forall l\in \lbrace 1,\ldots,n \rbrace, \forall \varepsilon\in \lbrace \pm 1 \rbrace,$ if $x+k\varepsilon e_i \in \mathring{V} \ \forall k\in \lbrace 1,\ldots,d \rbrace,$  then:
		$$f(x)\in \mathrm{Span}_{\Z}\lbrace f(x+k\varepsilon e_l) | 1\leq k \leq d \rbrace.$$
	\end{itemize}

	Then there exists a finite collection $(V_j)_{j\in J}$ of sub-vector spaces of $\Q^n$ of codimension at least one, such that
	$$\dim \mathrm{Span}_{\mathbb{Z}}\left\lbrace f(x) \ | \ x\in \Z^n\cap \left(\mathring{V} \setminus \underset{j\in J}{\cup} V_j \right) \right\rbrace \leq d^n.$$
\end{lemma}

\begin{proof}
	Let $V'\subset V$ be the subset of consisting of points that are at distance $\geq d+1$ to $\partial V$ for the infinity norm. Then $\Z^n \cap (V\setminus V')$ is a subset of a finite union of sets of the form $\lbrace x  \ | \ \varphi_i(x)=b_i \rbrace.$ Indeed, each $\varphi_i |_{\Z^n}$ take discrete values, and since any point $x$ of $V\subset V'$ is within bounded distance of $\partial V,$ one must have $|\varphi_i(x)-a_i|\leq M(d+1)$ for at least one $i\in I.$ (Here $M$ is the maximum of the norms of $\varphi_i$ subordinate to the norm $||\cdot ||_{\infty}$ in $\Q^n.$)
	
	We proceed to show that $\dim \mathrm{Span}_{\mathbb{Z}}\lbrace f(x) \ | \ x\in \Z^n\cap V' \rbrace \leq d^n.$ This is clear if $\Z^n\cap V'=\emptyset.$ Otherwise, pick $x\in \Z^n \cap V'.$ We claim that $\mathrm{Span}_{\mathbb{Z}}\lbrace f(x) \ | \ x\in \Z^n\cap V' \rbrace$ is spanned by the $\lbrace f(x+v_1e_1+\ldots v_n e_n) \ | 0\leq v_i \leq d-1 \rbrace.$ 
	
	Indeed, take $y\in \Z^n\cap V'.$ Notice that $V'$ is convex, hence the segment $[x,y]$ is entirely in $V'.$ Now, partition $\Q^n$ into cubes $x+dt+[0,d)^n,$ where $t\in \Z^n.$ We can find a sequence of cubes $C_1,\ldots ,C_k$ of the partition such that $C_1=x+[0,d)^n,$ $C_i$ and $C_{i+1}$ share a face, each $C_i$ is in $\mathring{V},$ and $[x,y] \subset \underset{1\leq i \leq k}{\cup} C_i.$ Those cubes are cubes that intersect $[x,y],$ ordered according to the first point of $[x,y]$ which they contain and, when tied, to the lexicographic order of their barycenter. 
	
	Now the lemma follows from the observation that if $C_i$ and $C_{i+1}$ share a face and are subsets of $\mathring{V},$ then $\mathrm{Span}_{\Z}\lbrace f(x) \ | \ x \in C_i \rbrace=\mathrm{Span}_{\Z}\lbrace f(x) \ | \ x \in C_{i+1} \rbrace,$ which is a direct consequence of condition (*).
\end{proof}

\section{The induction hypothesis and beginning of the induction proof}
\label{sec:induction1}
Let $M$ be a $3$-manifold with boundary, and $\lbrace \lambda_1,\ldots, \lambda_n \rbrace$ be a pair of pants decomposition of $\partial M.$ We also fix a Heegaard splitting $M=H\underset{\Sigma}{\cup} C,$ where $C$ is a compression body, and $H$ is a handlebody. We assume that $H$ has genus at least $2$, which can always be achieved via stabilization. Recall that $Sk(M,\Q(A))$ is spanned by the vectors $(\varphi_c)_{c \in \Delta}$ by Lemma \ref{lemma:basis}. We will show that $Sk^{\overline{\lambda}}(M)$ is a finite dimensional $\Q(A,\overline{\lambda})$ vector space by an induction. Our induction hypothesis is

\smallskip

\textbf{Induction hypothesis:} There exists a finite collection $(V_s)_{s\in S}$ of affine subspace of $\Z^{3g-3}$ of co-dimension at least $1,$ such that
$$\mathrm{Span}_{\Q(A,\overline{\lambda})}\lbrace \varphi_c \ | \ c\in \Delta \setminus \underset{s\in S}{\bigcup}V_s\rbrace \ \textrm{is finite dimensional}$$

where we induct lexicographically on the $3g-3$-tuple $(n_1,n_2,\ldots,n_{3g-3}),$ where $n_i$ is the number of affine subspace $V_s$ of codimension $i.$

\smallskip

We note that once arrived at $(n_1,\ldots,n_{3g-3})=(0,\ldots,0,k)$ for some $k\in \N,$ the finiteness theorem is proved, since subspaces of codimension $3g-3$ are just points. First we need to prove that some finite collection $(V_s)_{s\in S}$ exists.

\begin{proof}[Proof of the base case]
	For each $i\in \lbrace 1,\ldots, 3g-3 \rbrace,$ consider the family of curves $\lbrace \alpha_1,\ldots,\alpha_{3g-3},\beta_i \rbrace.$ By Lemma \ref{lem:linDep}, there is $P_i\neq 0$ in $\Q(A,\overline{\lambda})[x_1,\ldots,x_{3g-2}],$ such that $P_i(\alpha_1,\ldots,\alpha_{3g-3},\beta_i)=0$ in $Sk^{\overline{\lambda}}(C),$ where $C$ is the compression body in the Heegaard splitting of $M.$
	
	Therefore, we have 
	$$\forall c\in \Delta,  P_i(\alpha_1,\ldots,\alpha_{3g-3},\beta_i)\cdot \varphi_c =0 $$
	in $Sk^{\overline{\lambda}}(M).$
	
	Let $d_i$ be the degree of $P_i$ in $\beta_i.$
	By Lemmas \ref{lem:linDep} and \ref{lemma:curveActions}, for $c\in \Delta \setminus N(\partial \Delta),$ we have recurrence relations among the $\varphi_c \in Sk^{\overline{\lambda}}(M)$ of the form:
	
	$$R_{-2d_i}(A,A^{c_1},\ldots,A^{c_{3g-3}})\varphi_{c-2d_i\delta_i} + \ldots + R_{2d_i}(A,A^{c_1},\ldots,A^{c_{3g-3}})\varphi_{c+2d_i\delta_i}=0.$$
	
	In the above, the $(R_{i})_{-2d \leq i \leq 2d}$ are polynomials in $3g-2$ variables, with $R_{2d}$ and $R_{-2d}$ non-zero. Moreover, $\delta_i$ is the vector of $\Z^{3g-3}$ whose only non-zero coordinate is the $i$-th coordinate, which is $1.$ Also $N(\partial \Delta)$ designates a neighborhood of the boundary of $\Delta,$ that is, a finite union of translates of its facets.
	
	By Lemma \ref{lemma:non-vanishing_poly}, the coefficients $R_{-2d}(A,A^{c_1},\ldots,A^{c_{3g-3}})$ and $R_{2d}(A,A^{c_1},\ldots,A^{c_{3g-3}})$ are both non-zero on the complement in $\Z^{3g-3}$ of a finite union of codimension $\geq 1$ subspaces.
	
	Varying $i\in \lbrace 1,\ldots, 3g-3 \rbrace,$ we get that the $\varphi_c$ satisfy one-dimensional recurrence relations with invertible dominant coefficients in all directions, except when $c$ belongs to some finite union of codimension $\geq 1$ subspaces of $\Delta.$ The complement of this finite union is a finite union of chambers of $\Z^{3g-3}$ such that the map $c\longrightarrow \varphi_c$ satisfies the condition (*) of Lemma \ref{lemma:chamber}, for $d=max(4d_i).$
	
	Lemma \ref{lemma:chamber} can now be applied and gives that there exists a finite collection $(V_s)_{s\in S}$ of codimension $\geq 1$ subspaces of $\Z^{3g-3},$ such that 
	$$\mathrm{Span}_{\Q(A,\overline{\lambda})}\lbrace \varphi_c \ | \ c\in \Delta \setminus \underset{s\in S}{\bigcup}V_s\rbrace \ \textrm{is finite dimensional}$$ 
\end{proof}

Having proved the base case, we move on to proving the induction step. We divide the proof in two cases: 
\begin{itemize}
	\item[Case 1:] The collection $(V_i)_{i\in I}$ contains a subspace which is not contained in a translate of a facet of $\Delta.$
	\item[Case 2:] There is no such subspace.
\end{itemize} 
While the method of proof for the latter case applies also to the former, it is considerably more complicated. Hence for the clarity of exposition we will start by exposing the argument for the induction step in Case 1.

\begin{proof}[Proof of Induction step in Case 1]
	In all what follows, we set $n=3g-3.$
	
	Let $V_0$ be an element of our finite collection $(V_s)_{s\in S}$ of affine subspaces of $\Z^{3g-3},$ which is not contained in a translate of any facet of $\Delta,$ and which is of maximal dimension among those. Let $k=\dim V_0.$ First we choose a subset $J$ of cardinality $n-k$ of the canonical basis $\lbrace e_1,\ldots,e_n \rbrace$ of $\Z^n$ as in Lemma \ref{lemma:lin_alg}. There may be parallel copies of $V_0$ in the collection $(V_s)_{s\in S}.$ In that case, the parallel copies of $V_0$ must be of the form $V_0+x_{j_1}e_{j_1}+\ldots + x_{n-k}e_{j_{n-k}},$ where the $x_j$'s are in $\Z.$ We will now redefine $V_0$ to be a parallel copy $V_0+x_{j_1}e_{j_1}+\ldots + x_{n-k}e_{j_{n-k}},$ where $(x_{j_1},\ldots,x_{j_{n-k}})$ is maximal for the lexicographical ordering.
	
	We pick another subset $I\subset \lbrace 1,\ldots,n \rbrace$ according to Lemma \ref{lemma:non-vanishing_poly}. Now for any $l \in \lbrace 1,\ldots n \rbrace \setminus J,$ consider the family of curves 
	$$F_l=\lbrace \alpha_i\rbrace_{i\in I} \cup \lbrace \beta_j \rbrace_{j\in J \cup \lbrace l \rbrace}.$$
	The family $F_l$ contains $n+1=3g-2$ elements, hence we can apply Lemma \ref{lem:linDep} to get a non-zero polynomial $P_l$ such that
	$$P_l(\alpha_i,\beta_j,\beta_l)=0 \in Sk^{\overline{\lambda}}(C),$$
	and therefore $P_l(\alpha_i,\beta_j,\beta_l)\cdot \varphi_c=0$ in $Sk^{\overline{\lambda}}(M),$ for any $c\in \Delta.$
	
	Let $D_l$ be the total degree of $P_l$ in the variables $\beta_j,\beta_l.$ We consider $U\subset V_0,$ which is the subset of $V_0$ consisting of elements that are at distance at least $4D_l$ to any facet of $\partial \Delta$ and to any other subspace $V_s$ where $s\in S$ which is not parallel to $V_0.$ Since all the subspaces considered are non-parallel to $V_0,$ the set $U$ is the complement in $V_0$ of a finite union of subspaces of $V_0$ of codimension at least $1.$  
	
	With this setup, we have:
	
	\medskip
	
	\textbf{Claim:} For any $l\in \lbrace 1,\ldots,n \rbrace \setminus J$ and $c\in U,$ we have a recurrence relation for the $\varphi_c$ of the form 
	
	$$ R_{d_l}(A,A^{c_{i_1}},\ldots , A^{c_{i_k}})\varphi_{c+d_l u_l} + \ldots + R_{-d_l}(A,A^{c_{i_1}},\ldots , A^{c_{i_k}})\varphi_{c-d_l u_l} \equiv 0,$$
	where the vector $u_l$ is a generator of the line $D_l$ given in Lemma \ref{lemma:lin_alg}, $d_l$ is an integer, $R_{\pm d_l}$ are non-zero polynomial in $k+1$ variables, and $\equiv$ is equality up to an element of 
	$$E_0=\mathrm{Span}_{\Q(A,\overline{\lambda})}\lbrace \varphi_c \ | \ c \in \Delta \setminus \underset{s\in S}{\bigcup} V_s \rbrace.$$
	
	\begin{proof}[Proof of the Claim] Consider $N(P_l),$ the Newton polygon of the polynomial $P_l$ in terms of the variables $\beta_j,\beta_l,$ which we define to be the convex hull of the elements $\underset{j\in J \cup \lbrace l \rbrace}{\sum}2\varepsilon_j n_j \delta_j,$ ranging over all $n+1-k$-uples $(n_j)_{j\in J\cup \lbrace l \rbrace}$ such that $\underset{j\in J\cup \lbrace l \rbrace}{\prod}\beta_j^{n_j}$ is a monomial in $P_l$ with non-zero coefficient. Then, for any $c$ such that $c+N(P_l)\subset \Delta,$ we have
		
		\begin{equation} \label{eq:rec_rel-Case1}P_l(\alpha_i,\beta_j ,\beta_l)\cdot \varphi_c= \underset{x\in N(P_l)}{\sum }R_x(A,A^{c_{i_1}},\ldots A^{c_{i_k}})\varphi_{c+x}.\end{equation}
		
		where the $R_x$ are rational fractions with non-vanishing denominators at $c$, and furthermore $R_x$ is a non-zero if $x$ is a vertex of $N(P_l).$
		
		However, thanks to Lemma \ref{lemma:lin_alg}, there is $v\in \Z^n$ such that $c+v+N(P_l)$ intersects $V_0$ only in an affine line parallel to $D_l,$ and intersects no subspace $V_s\neq V_0$ which is  parallel to $V_0.$ Indeed, pick $v=y_{j_1}e_{j_1}+\ldots y_{j_{n-k}}e_{j_{n-k}}$ with $y$ maximal for the lexicographical order such that $V_0$ and $c+v+N(P_l)$ intersect; then $c+v+N(P_l)$ does not intersect other $V_s$ parallel to $V_0$ by maximality, and it intersects $V_0$ in a point or in a line parallel to $D_l,$ by construction of the set $J.$ We can further add to $v$ an element of the direction of $V_0$ to recenter this line. In the end, applying Equation \ref{eq:rec_rel-Case1} to $c'=c+v,$ we get the claim, since all other elements of $c+v+N(P_l)$ are in $E_0,$ by construction of $U.$
	\end{proof}
	
	Given the claim, we conclude similarly to the proof of the base case: there is a finite collection of codimension $\geq 1$ subspaces in $V_0$ such that all polynomials $R_{\pm d_l}$ are non-vanishing on the complement, and those subspaces divide $U$ into a finite collection of chambers to which we can apply Lemma \ref{lemma:chamber}. 
	
	Note that since the recurrence relations involved the relation $\equiv$ instead of equality, what we get from this is the following. There is a finite collection $(W_t)_{t\in T}$ of codimension $\geq 1$ subspaces in $V_0,$ such that, writing 
	$$E_1=\mathrm{Span}_{\Q(A,\overline{\lambda})}\lbrace\varphi_c \ | \ c \in V_0 \setminus \underset{t\in T}{\bigcup} W_t \rbrace$$
	and 
	$$E_0=E_1 \cap \mathrm{Span}_{\Q(A,\overline{\lambda})}\lbrace \varphi_c \ | \ c \in \Delta \setminus \underset{s\in S}{\bigcup} V_s \rbrace,$$ 
	then we have
	$$\dim E_1/E_0 <+\infty.$$
	
	However, $\dim E_0<+\infty$ by induction hypothesis, hence $\dim E_1<+\infty.$ 
	
	Therefore, we can replace in the finite collection $(V_s)_{s\in S},$ the subspace $V_0$ by the finite collection $(W_t)_{t\in T},$ which contains only subspaces of higher codimension than $V_0.$ Therefore we have completed the induction step in Case 1. 
\end{proof}

The proof of the inductive step in Case 2, which is significantly more involved, will be completed in Section \ref{sec:finalproof}, after we construct a set of operators to replace the $\beta_i$s in the proof of Case 1.

\section{The boundary of polytope of admissible colorings}
\label{sec:polytope-boundary}
\subsection{Notation and basic definitions}

\begin{definition}\label{def:nearadmissible}
	Let $\Gamma$ be a lollipop tree. Take some edges $e_1,\dots,e_l$, some directed vertices $(v_{l+1},e_{l+1}),\dots,(v_{k},e_k)$ and natural numbers $a_1,\dots,a_k$. 
	
	Then the \emph{barely admissible subspace} of $\Delta$ associated to these choices is the subset $V\subseteq \Delta$ defined by the equations
	$d(e_i)=a_i$ for $1\leq i\leq l$, and $d(v_{j},e_{j})=a_j$ for $l+1\leq j\leq k$.
	
	We say that a subset $V\subseteq \Delta$ is barely admissible if it is the barely admissible subspace associated to some choice.
	
\end{definition}

\begin{remark}
	A barely admissible subspace is not, technically speaking, a subspace; it is rather a subspace of $\Z^\mathcal{E}$ intersected with $\Delta$. With a slight abuse of notation, we will talk about dimension, codimension or a basis for a barely admissible subspace $V$; we will always mean dimension, codimension or basis for the unique subspace of $\Z^\mathcal{E}$ whose intersection with $\Delta$ gives $V$.
\end{remark}

In what follows, when we write $d(x)$ we mean that $x$ is either an edge or a directed vertex. If we wish to specify that $x$ is an edge we will use the letter $e$ instead. 

\begin{definition}\label{def:simpleNAS}
	Let $V$ be a barely admissible subspace characterized by the equations $\{d(x_i)=a_i\}$ for $i=1,\dots,k$. Then $F$ is \emph{simple} if the following conditions hold:
	\begin{itemize}
		\item if $x_i=e_i$ and $x_j=e_j$ then $e_i=e_j$ implies $i=j$;
		\item if $x_i=(v_i,e_i)$ and $x_j=(v_j,e_j)$, $v_i=v_j$ implies $i=j$;
		\item if $v$ is a vertex of $\Gamma$ that does not appear among the $x_i$'s, then either all edges adjacent to $v$ appear among the $x_i$'s or none of them do;
		\item if $x_i=(v_i,e_i)$ and two of the edges adjacent to $v$ coincide, they must not be $e_i$;
		\item if $x_i=(v,e)$, then $e$ does not also appear as an edge among the $x_i$s;
		\item if $x_i=(v,e)$, then at most one edge adjacent to $v$ appears as an edge among the $x_i$s;
		\item if $x_i=(v,e)$ and a loop edge is adjacent to $v$, it must not appear among the $x_i$s.
	\end{itemize}
\end{definition}
\begin{lemma}\label{lemma:unionsimpleNAS}
	A non-empty barely admissible subspace $V$ is a finite union of simple barely admissible subspaces.
\end{lemma}
\begin{proof}
	Call $x_i$ and $a_i$ the choices used to define $V$. Notice that if $x_i=x_j=e_i$ then either $V$ is empty or one equation is redundant, hence can be tossed out.
	
	Suppose that there are $i\neq j$ such that $v_i=v_j$; if $e_i=e_j$ and $a_i=a_j$ then the equation is redundant and can be tossed out, whereas if $a_i\neq a_j$ then $V$ is empty. So suppose that $e_i\neq e_j$. Then two of the equations characterizing $V$ are $d(e_k)+d(e_j)-d(e_i)=a_i$ and $d(e_k)+d(e_i)-d(e_j)=a_j$, where $e_k$ is the third edge adjacent to $v_i$. Summing these two equations together we find that $2d(e_k)=a_i+a_j$, which implies either $V$ empty (if $a_i+a_j$ is odd) or $d(e_k)=(a_i+a_j)/2$. Therefore we can toss out one among $x_i,x_j$ and substitute it (without changing $V$) with $e_k$ and $a_k:=(a_i+a_j)/2$. Repeat this process until the choices for $V$ satisfy the second condition of simplicity.
	
	Now suppose that $v$ does not appear as a vertex among the $x_i$'s, and call $e_1,e_2,e_3$ the edges adjacent to $v$. Suppose that $e_1$ appears among the $x_i$'s and the other two do not. Then the triangular inequalities defining $\Delta$ say that $d(e_2)+a_1-d(e_3)\geq 0$ and $d(e_3)+a_1-d(e_2)\geq 0$, which implies in particular that $d(e_2)-d(e_3)\leq a_1$. Therefore we can add to the equations defining $V$ the inequality $0\leq d(e_2)+d(e_1)-d(e_3)\leq 2a_1$ without changing $V$. Therefore $V=V_1\sqcup \dots V_{a_1}$ where $V_k$ is defined by the same equations as $V$ plus the additional equation $d(v,e_3)=2k$.
	
	Now take $v$ any vertex, call as before $e_1,e_2,e_3$ the adjacent edges and suppose that $e_1$ and $e_2$ appear among the $x_i$'s and $e_3$ does not. Then the triangular inequalities at $v$ imply that $e_3$ must be at most $a_1+a_2$, therefore we can decompose $V$ into subspaces where $e_3$ appears among the $x_i$'s.
	
	Suppose that $x_i=(v_i,e_i)$ and $e_i$ coincides with another edge adjacent to $v_i$; then the equation given by $x_i$ reduces to $d(e)=a_i$, where $e$ is the remaining edge; thus we can replace $x_i$ with $e$.
	
	Now suppose that $v$ appears as a vertex among the $x_i$'s, with chosen edge $e$, and suppose that $e$ appears among the $x_i$'s. Then in $V$ we have $d(v,e)=a$ and $d(e)=b$, which implies that $d(e_1)+d(e_2)= a+b$, and since both $d(e_1)$ and $d(e_2)$ are non-negative, $V$ is contained in a finite union of barely admissible subspaces where $e_1$ and $e_2$ are both among the $x_i$'s. 
	
	Finally suppose that $v$ appears as a vertex among the  $x_i$'s, with chosen edge $e$, and both of the other edges adjacent to $v$ appear among the $x_i$'s. Then this determines the value of $e$ in $V$, hence we can replace $v$ with $e$ without changing $V$. This also applies if a loop edge adjacent to $v$ appears among the $x_i$s.
\end{proof}

\subsection{The singular and flow subgraphs}

\begin{figure}
	\includegraphics[scale=0.3]{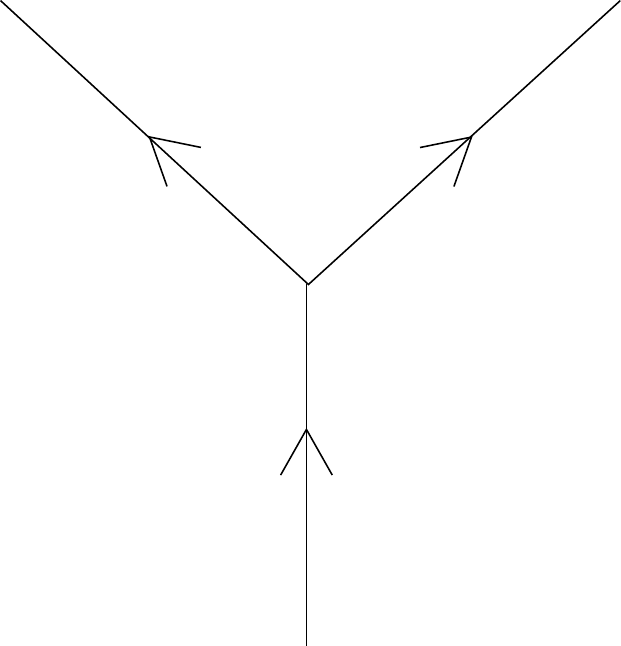}
	\caption{The direction of half edges around a vertex of $\Phi_V$; the upwards edge is the one appearing with a negative sign in the defining equations of $V$}
	\label{fig:directededge}
\end{figure}
\begin{definition}\label{def:subgraphs}
	If $V$ is a simple barely admissible subspace characterized by the equations $\{d(x_i)=a_i\}$, we can associate to it two subgraphs of $\Gamma$, the singular subgraph $\Gamma_V$ and the flow subgraph $\Phi_V$.
	
	The \emph{singular subgraph} is defined as the union of all edges appearing among the $x_i$'s; we will draw it with dashed edges.
	
	The \emph{flow subgraph} $\Phi_V$ is defined as the closure of the union of the non-singular edges adjacent to each vertex appearing among the $x_i$'s; the signs of each adjacent edge of $v$ gives an orientation on the half edges adjacent to $v$ as in Figure \ref{fig:directededge}. This makes $\Phi_V$ into a bidirected graph; we draw it depicting the orientations.
	
	Denote with $R_V$ the closure of the complement of $\Gamma_V\cup\Phi_V$ in $\Gamma$, and call it the regular subgraph of $V$.
	
\end{definition}

We can define a partial order on the vertices of the flow subgraph; we say that $v_1\succ v_2$ if there is a directed path from $v_1$ to $v_2$, and extend it to a total order on the $x_i$'s (so that all the vertices appear after the edges); call this order $\succ$.

\begin{lemma}\label{lemma:codimension}
	If $V$ is a simple barely admissible subspace characterized by the equations $\{d(x_i)=a_i\}$ for $i=1,\dots,k$, then its codimension is $k$.
\end{lemma}
\begin{proof}
	Since $V$ is realized as an intersection of affine hyperplanes, we need only to show that the dual vectors $x_i$ to the supporting hyperplanes are linearly independent in $\Z^\mathcal{E}$. These vectors are of the form $e_i$ (if the hyperplane is of the form $d(e_i)=a_i$) or $e_i+e_j-e_k$ (if the hyperplane is of the form $d(v_i,e_k)=a_i$). Suppose that $\lambda:=\sum_i\lambda_i x_i=0$; if we show that in any such sum, there is one edge $e$ with non-zero coefficient which only appears once, we have our desired result (since the edges themselves are linearly independent).
	
	To see this, first consider $\Phi'$ the subgraph of the flow subgraph $\Phi_V$ spanned by the vertices with non-zero coefficient in $\lambda$. If this is non-empty, take $e$ a leaf of $\Phi'$; it cannot be a dashed edge (since they are not part of $\Phi_V$) hence it can only appear in $\lambda$ once.
	
	Suppose instead that $\Phi'$ is empty; then any edge in $\lambda$ corresponds to an edge of the singular subgraph and thus can only appear once in $\lambda$.
	
\end{proof}

From now on, we fix $V$ a codimension $k$ simple barely admissible subspace of $\Delta$; we call $x_1,\dots,x_k$ the choices of edge or directed vertex defining $V$, ordered according to $\succ$. We also call $f:\Z^\mathcal{E} \ra\Z^{k}$ the function given by $f(d)=(d(x_1),\dots,d(x_k))$. We consider $\Z^k$ to be ordered under the lexicographic order induced by $\succ$. 

Consider the set $\mathcal{E}'$ given by:

\begin{itemize}
	\item The regular edges of $\Gamma$;
	\item the introverted edges of $\Phi_V$;
	\item the leaves of $\Phi_V$.
\end{itemize}

Notice that the cardinality of $\mathcal{E}'$ is larger than $\textrm{dim}(V)$; more precisely the dimension of $V$ is equal to the cardinality of $\mathcal{E}'$ minus the number of extraverted edges of $\Phi_V$ (this will be proved in Lemma \ref{lemma:basisV}).

We define the \emph{orthogonal map} as the map $g:\Z^\mathcal{E}\ra \Z^{\mathcal{E}'}$ given by $g(\delta_e)=\delta_e$ if $e$ is an introvorted edge of $\Phi_V$, a non-loop leaf of $\Phi_V$, or a regular edge, $g(\delta_e)=2\delta_e$ if $e$ is a loop edge of $\Phi_V$ and $0$ otherwise.

Finally we fix an additional ordering $\succ'$ that we will use to select certain edges later. Suppose $e_1,\dots,e_l$ are a maximal collection of introverted edges of $\Phi_V$ such that $e_i\cap e_{i+1}=v_i$; then we say that $e_i\succ'e_j$ if $i<j$. Notice that since there are no cycles in $\Phi_V$ this gives a total order on the $e_i$s. Extend this order arbitrarily to an order, still denoted $\succ'$, of all regular edges of $\Gamma$, introverted edges and leaves of $\Phi_V$.

\subsection{Operators adapted to barely admissible subspaces}
\label{sec:curvop-nearadm}
In this subsection we build some operators that, in some sense, behave well with respect to a given barely admissible subspace; we will use them to, roughly speaking, push a color in $V$ further into $\Delta$. We build these operators from the singular and flow subgraphs of $\Gamma$, using the curve operators of Lemmas \ref{lemma:curveActions} and \ref{lemma:longcurveactions}. We first start with a very basic lemma about curve operators.

\begin{lemma}\label{lemma:minshiftdashed}
	Let $e\subseteq \Gamma_V$ be a dashed edge; then the operator $T^{\beta_e}$ has a unique minimal shift given by $-\delta_e$ if $e$ is a loop edge, and $-2\delta_e$ otherwise. 
\end{lemma}
\begin{proof}
	By Lemma \ref{lemma:curveActions}, the shifts in $T^{\beta_e}$ are $0$ and $\pm 2\delta_e$, and $f(\delta_e)=\delta_e$, therefore the unique minimal shift of $T^{\beta_e}$ is $-2\delta_e$.
\end{proof}

The problem of using the operator $T^{\beta_e}$ for our purposes is that, if $\Gamma_V$ contains connected components of more than one edge, there will be no admissible color $c$ such that $\varphi_{c-2\delta_{e}}\in V$, or in other words, no element of $V$ is effective for $T^{\beta_e}$. Therefore we need to use more complicated operators that are adapted to $V$.

\begin{definition}\label{dfn:pathop}
	Consider a path $\gamma\subseteq \Gamma$. We define the \emph{path operator} $T^{\gamma}$ as $$T^{\gamma}:=T^{\beta_{\mathring{\gamma}}}T^{\beta_{\partial\gamma}}.$$
\end{definition}

The next lemmas act as a definition of the operators we are looking for; the proofs are all very similar so we will carry out in detail the first, and then simply point out the subtle differences for the others.

\begin{lemma}\label{lemma:extopfirst}
	Let $\gamma\subseteq \Gamma_V$ be a maximal arc, and consider the path operator $T^{\gamma}:\Sk(H)\ra\Sk(H)$.
	
	Then:
	\begin{itemize}
		\item $T^\gamma$ has a unique minimal shift $x$;
		\item $f(x)(e)=-1$ if $e$ is a loop edge of $\gamma$, $f(x)(e)=-2$ if $e$ belongs to $\mathring{\gamma}$, $f(x)(v)=-2$ if $v$ is an endpoint of $\gamma$, $f(x)(y)=0$ otherwise and $g(x)=0$.
		\item there is a finite union of barely admissible subspaces $W\subseteq V$, each of codimension $1$, such that for all $y\in V\setminus W$, $y$ is effective for $T^\gamma$ and $y-x\in \Delta\setminus V$.
	\end{itemize}
\end{lemma}

\begin{proof} By Lemmas \ref{lemma:longcurveactions} and \ref{lemma:uniqueminimalshift}, $T^{\gamma}$ has a unique minimal shift given by $x=-\sum_{e \textrm{ in }\partial\gamma} \delta_e-2\sum_{e \textrm{ in }\mathring{\gamma}} \delta_e$. The values of $f(x)$ and $g(x)$ follow at once. We need to prove the last property.
	
	Take $y\in V$; we know that clearly $y-x\notin V$, therefore we want to show that $y-x\in \Delta$ unless $y$ belongs to $W$, where $W$ is some union of codimension $1$ barely admissible subspaces, i.e. the color $y-x$ satisfies the triangular inequalities (since the non-negative and evenness conditions are automatic, as $x$ has negative entries which are even on non-loop edges). Consider the triangular inequalities corresponding to a vertex $v$. If $\gamma$ does not contain $v$, then clearly if $y$ satisfies the triangular inequality at $v$ then so does $y-x$. If $\gamma$ contains $v$ in its interior, then the same holds. Indeed, if $v$ is not adjacent to a loop edge, $-x$ is adding $2$ to two edges out of three adjacent to $v$; this means that either $2$ is added to both sides of the inequality, or $4$ is added to the larger side. If $v$ is adjacent to a loop edge then $2$ is added to both sides of the inequality. Finally, if $v$ is an endpoint of $\gamma$, then $-x$ is adding $2$ to one side of the inequality; the only case to consider is if it is adding it to the smaller side. If $y$ is such that $y(e_i)+y(e_j)\geq y(e_l)$ but $y(e_i)+y(e_j)\ngeq y(e_l)+2$, then it means that $y(e_i)+y(e_j)=y(e_l)$ and therefore $y$ belongs to the subspace of $V$ given by this extra equation (notice that the subspace is proper because of the simplicity assumption). Repeating this reasoning for the other endpoint (if any) of $\gamma$ shows that if $x$ does not belong to the union of these two barely admissible subspaces, $y-x\in \Delta$. The fact that $y$ is effective for $T^\gamma$ follows from Lemma \ref{lemma:longcurveactions} and Corollary \ref{cor:effectiveness}.
\end{proof}

\begin{lemma}\label{lemma:extopsecond}
	Let $v\in\Phi_V$ be a vertex and $\gamma_v$ be a maximal directed path starting from $v$. Notice that the final edge of $\gamma_v$ is either a leaf of $\Phi_V$ or an introverted edge. Define the path operator
	$T^{v}:=T^{\gamma_v}$.
	Then
	
	\begin{itemize}
		\item $T^{v}$ has a unique minimal shift $x$;
		\item $f(x)(w)=-2$ if $w$ is an endpoint of $\gamma_v$ in $\Phi_V$,  $f(x)(v')=0$ if $v'$ is not, and $f(x)(e)=0$;
		\item $g(x)(e)=-2$ if $e$ is the final edge of $\gamma_v$, $0$ otherwise.
		\item there is a union of barely admissible subspaces $W\subseteq V$, each of codimension $1$, such that for all $y\in V\setminus W$, $y$ is effective for $T^\gamma$ and $y-x\in \Delta\setminus V$.
	\end{itemize}
\end{lemma}
\begin{proof}
	The fact that $T^{v}$ has a unique minimal shift equal to $$x=-\sum_{e \textrm{ edge of }\partial\gamma_v}\delta_e-2\sum_{e \textrm{ edge of }\mathring{\gamma}_v} \delta_e$$ is a consequence of Lemmas \ref{lemma:curveActions} and \ref{lemma:uniqueminimalshift}. This means that if $w$ is an endpoint of $\gamma_v$, $w$ has only one adjacent edge in the sum, thus $f(x)(w)=-2$. For every other vertex $v'\in \Phi_V$, either $v'\notin \gamma_v$, or exactly two edges in $\gamma_v$ are adjacent to $v'$ and they appear with opposite signs, or $v'$ is in $\gamma_v$ and adjacent to a loop edge; in all cases $f(x)(v')=0$. The other calculations about $f$ and $g$ are obvious.
	
	The proof of the last part of the result follows from the same line of reasoning as Lemma \ref{lemma:extopfirst}.
	
\end{proof}

\begin{lemma}\label{lemma:intopfirst}
	Let $e\subseteq R_V$ be a regular edge, and define $T^e:=\left(T^{\beta_e}\right)^2$ if $e$ is a loop edge, and $T^e:=T^{\beta_e}$ otherwise. Then 
	
	\begin{itemize}
		\item $T^e$ has two different non-zero minimal shifts $2\delta_e$ and $-2\delta_e$;
		\item there is a union of barely admissible subspaces $W\subseteq V$, each of codimension $1$, such that for all $y\in V\setminus W$ and any $x$ minimal shifts of $T^e$, $y$ is effective for $T^e$ and $y-x\in V$.
	\end{itemize}
\end{lemma}
\begin{proof}
	The first part of this Lemma is simply noting that the three shifts of $T^e$ given by Lemma \ref{lemma:curveActions} are all in the kernel of $f$. The second part is proved exactly the same way as Lemma \ref{lemma:extopfirst}.
\end{proof}
\begin{remark}
	The only reason for the square in the definition of $T^e$ is to simplify notation in later lemmas.
\end{remark}

\begin{lemma}\label{lemma:intopsecond}
	Let $e$ be a leaf of $\Phi_V$ and $\gamma_e\subseteq \Phi_V$ be the maximal directed path ending in $e$, and define the path operator $T^{\gamma_e}$. Then 
	
	\begin{itemize}
		\item $T^{\gamma_e}$ has a unique minimal shift $x$;
		\item if the first edge of $\gamma_e$ is directed, $f(x)=-4\delta_v$ where $v$ is the second vertex of $\gamma_e$;
		\item if the first edge of $\gamma_e$ is extraverted, $f(x)=-4(\delta_v+\delta_{v'})$, where $v,v'$ are the endpoints of the extraverted edge.
		\item $g(x)=-2\delta_e$;
		\item there is a union of barely admissible subspaces $W\subseteq V$, each of codimension $1$, such that for all $y\in V\setminus W$, $y$ is effective for $T^\gamma$ and $y-x\in \Delta$.
	\end{itemize}
\end{lemma}
\begin{proof}
	The fact that $T^{\gamma}$ has a unique minimal shift works exactly the same way as before. We assume for simplicity that $\gamma$ contains no loop edges; the other case is identical. By Lemma \ref{lemma:longcurveactions}, the shifts of $T^{\gamma}$ are $$\sum_{e \textrm{ edge of }\gamma}\pm 2\delta_{e}.$$
	
	Number the internal vertices of $\gamma$ as $v_1\succ v_2 \succ \dots \succ v_r$, and number the edges of $\gamma$ as $e_0,\dots,e_r$ in the same order. If the first edge of $\gamma$ is an extraverted edge, then $\gamma$ has as endpoint an additional $v_0$ in the interior of $\Phi_V$; in this case $v_0$ might be larger than $v_1$ but either way the proof proceeds the same way.
	
	Suppose first that either $v_1\succ v_0$ or $v_0$ does not exist. Then to find the minimal shift $x$ we need to first minimize $f(x)(v_1)$ which is achieved by choosing the $+$ sign on $e_0$ and the negative sign on $e_1$. Among these shifts, we then need to minimize $f(x)(v_2)$, which is achieved by choosing the $-$ sign on $e_3$, and so on until we find that the minimal shift must be $2\delta_{e_0}-2\sum_{i=1}^r\delta_{e_i}$.
	
	If instead $v_0\succ v_1$, we first must minimize $f(x)(v_0)$ which is achieved by choosing the $+$ sign on $e_0$ anyways, which means that the rest of the argument goes exactly the same and the conclusion is the same.
	
	The calculations of the values of $f$ and $g$ are straightforward.

	The proof for the last bullet point is exactly the same as the one for Lemma \ref{lemma:extopfirst}; notice that in this case, even though the minimal shift of $T^{\mathring{\gamma}}$ has some positive coefficients, the minimal shift of $T^{\partial\gamma}$ is non-positive, and that is all we need to apply Corollary \ref{cor:effectiveness}.
\end{proof}

The next lemma is simply a variation of Lemma \ref{lemma:intopsecond} and is proven in the exact same way.
\begin{lemma}\label{lemma:intopsecondintro}
	Let $e$ be an introverted edge of $\Phi_V$, $\gamma_e\subseteq \Phi_V$ be a maximal directed path ending in $e$, and define the path operator $T^{\gamma_e}$. Then 
	
	\begin{itemize}
		\item $T^{\gamma_e}$ has a unique minimal shift $x$;
		\item if the first edge of $\gamma_e$ is directed, $f(x)=-4\delta_v-2\delta_{w}$ where $v$ is the second vertex of $\gamma_e$ and $w$ is the endpoint of $\gamma_e$ on the introverted edge;
		\item if the first edge of $\gamma_e$ is extraverted, $f(x)=-4(\delta_v+\delta_{v'})-2\delta_w$, where $v,v'$ are the endpoints of the extraverted edge and $w$ is the endpoint of $\gamma_e$ on the introverted edge.
		\item $g(x)=-2\delta_e$;
		\item there is a union of barely admissible subspaces $W\subseteq V$, each of codimension $1$, such that for all $y\in V\setminus W$, $y$ is effective for $T^\gamma$ and $y-x\in \Delta$.
	\end{itemize}
\end{lemma}

\section{Witten's finiteness conjecture}
\label{sec:finalproof}

\subsection{External and internal operators of $V$}

We now proceed to define a set of operators that will allow us to perform the inductive step in the proof.

\begin{definition}\label{def:extopfirst}
	Take a connected component of $\mathring{\Gamma}_V$, take a planar projection of it and choose one if its leaves $e$ as a root. Notice that if this connected component contains $2n+1$ edges, then it contains $n+2$ leaves. Take $\gamma_1,\dots,\gamma_{n+1}$ maximal arcs each connecting $e$ to one of the other leaves, and also take $\gamma_{n+2},\dots,\gamma_{2n+1}$ maximal arcs connecting two adjacent leaves (different from $e$) in the chosen projection. Append to each $\gamma_i$ any edge loop in $\Gamma_V$ that it touches. Repeat this process for each connected component of $\Gamma_V$ to obtain $\gamma_1,\dots,\gamma_l$ where $l$ is the number of edges of $\mathring{\Gamma}_V$, and take $T^{e_{l+1}},\dots,T^{e_{r}}$ where $e_k$ is an edge loop of $\Gamma_V$. We call $T^{\gamma_1},\dots,T^{\gamma_l},T^{e_{l+1}},\dots,T^{e_r}$ the \emph{external operators of the first kind} for $V$.
\end{definition}
\begin{lemma}\label{lemma:exttopfirstindep}
 The external operators of the first kind have unique minimal shifts $x_1,\dots,x_r$ and $f(x_1),\dots,f(x_r)$ are linearly independent.
\end{lemma}
\begin{proof}
	The uniqueness of the minimal shifts is from Lemma \ref{lemma:extopfirst}; we only need to show that $f(x_1),\dots,f(x_r)$ are linearly independent. To do this, we show by induction on $r$ that the span over $\Q$ of $f(x_1),\dots,f(x_r)$ is equal to the span of $\delta_{e_1},\dots,\delta_{e_r}$, where $e_1,\dots,e_r$ are the edges of $\Gamma_V$. If there are edge loops in $\Gamma_V$, then some of the $f(x_i)$'s are equal to $\delta_{e_i}$; we can simply remove them from any $f(x_j)$ where they appear without affecting the validity of the statement. Therefore we can assume that there are no edge loops. We now proceed by induction on $r$.
	
	\begin{figure}
		\includegraphics[scale=0.4]{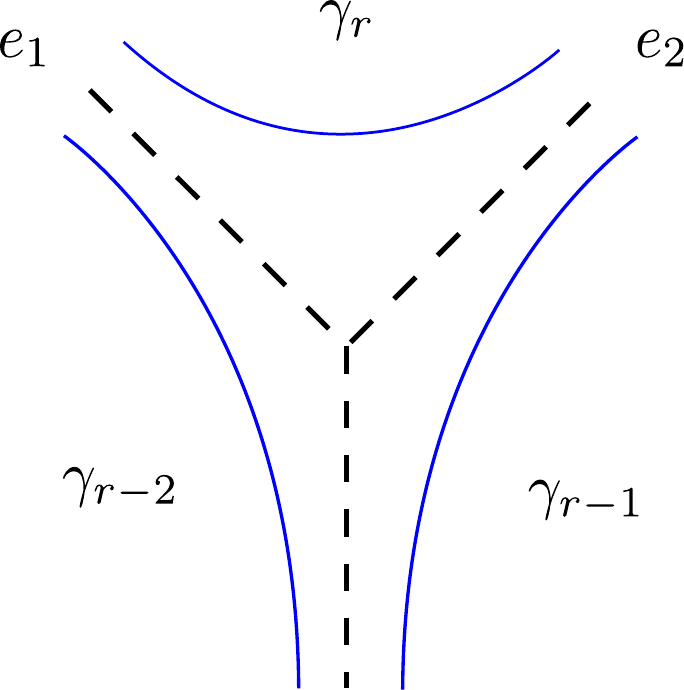}
		\label{fig:extoperators}
		\caption{The three paths touching $e_1$ and $e_2$.}
	\end{figure}
	
	If $r=1$ there is nothing to prove. Suppose $r>1$ (which also means $r>2$ by the simplicity assumptions) and take $e_1,e_2$ (both different from the root $e$) two leaves of $\Gamma_V$ sharing a vertex. There are exactly three $\gamma_i$'s touching either $e_1$ or $e_2$, depicted in Figure \ref{fig:extoperators}; the two arcs, say (without loss of generality) $\gamma_{r-2},\gamma_{r-1}$, connecting $e_1,e_2$ respectively to $e$ and the arc connecting the two of them (which is simply $\gamma_r:=e_1\cup e_2$). Then we can obtain $4\delta_{e_1}:= x_{r-2}-x_{r-1}+x_r$ and $4\delta_{e_2}:= x_{r-1}-x_{r-2}+x_r$. This means that the span of $f(x_1),\dots,f(x_r)$ is equal to the span of $f(x_1),\dots,f(x_{r-3}),f(x_{r-2})-2\delta_{e_1},\delta_{e_1},\delta_{e_2}$. The first $r-2$ vectors of this list are the minimal shifts of the external operators associated to the singular graph obtained by removing $e_1,e_2$ from $\Gamma_V$, thus we can apply the inductive hypothesis and obtain the lemma.
\end{proof}

\begin{definition}\label{def:extopsecond}
Fix a planar embedding of $\Phi_V$. 
For any vertex $v$ of $\Phi_V$, we define a path $\gamma_v$ recursively as follows.
\begin{itemize}
	\item 
	If both outgoing edges of $v$ are leaves, define $\gamma_v$ as the leftmost edge of $v$;
	\item 
	If both outgoing edges of $v$ are introverted edges, define $\gamma_v$ as the larger outgoing edge;
	\item if one outgoing edge of $v$ is a leaf and the other is introverted, define $\gamma_v$ as the leaf;
	\item if at least one outgoing edge of $v$ is directed, choose $w$ the other endpoint of the leftmost outgoing edge and define $\gamma_v:=\gamma_w\cup \overrightarrow{vw}$.
\end{itemize}

\begin{figure}
	\includegraphics[width=0.5\textwidth]{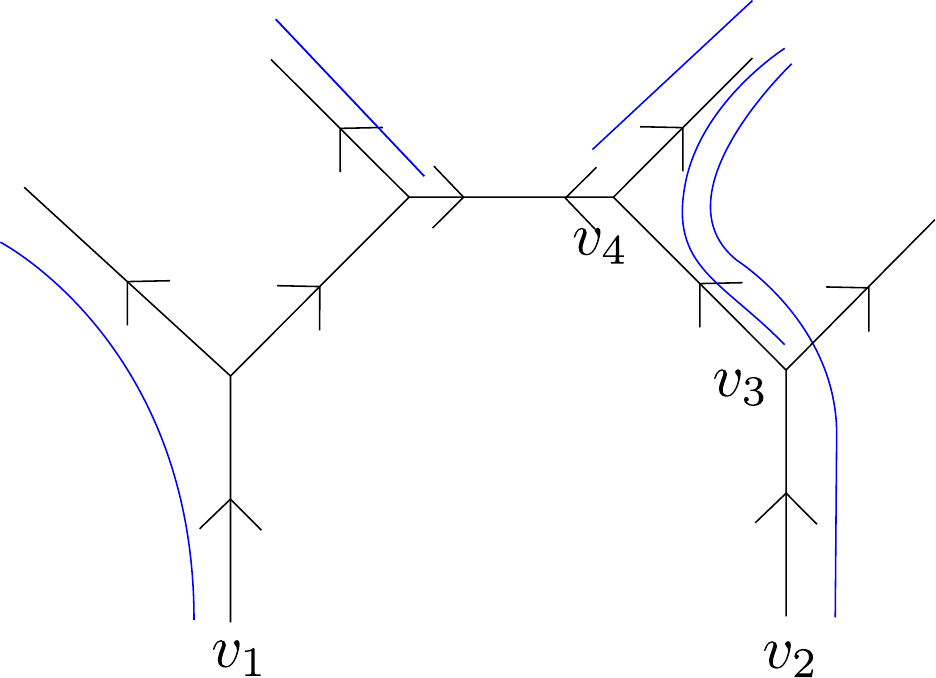}
	\label{fig:externalpaths}
	\caption{The paths $\gamma_{v_i}$ in Definition \ref{def:extopsecond}; notice that they always turn left except to avoid introverted edges.}
\end{figure}

Some examples are depicted in Figure \ref{fig:externalpaths}
 
 Let $v_1\succ \dots\succ v_n$ be the vertices of $\Phi_V$;  and define $T^{v_i}=T^{\gamma_{v_i}}$ as in Lemma \ref{lemma:extopsecond}. We call $T^{v_1},\dots,T^{v_n}$ the external operators of the second kind for $V$.
\end{definition}
\begin{lemma}\label{lemma:exttopsecondindep}
	The external operators of the second kind have unique minimal shifts $x_1,\dots,x_n$ and $f(x_1),\dots,f(x_n)$ are linearly independent.
\end{lemma}
\begin{proof}
	Given $v_i$, $T^{v_i}$ has a unique minimal shift $x_i$ given by Lemma \ref{lemma:extopsecond}. Consider a linear relation $X=\sum_i \lambda_i x_i=0$. Suppose at first that $e_0,\dots,e_r$ are introverted edges of $\Phi_V$ that form a connected subset $\Lambda$ of $\Phi_V$, and call $v_1,\dots,v_r$ the internal vertices of $\Lambda$ (with the convention that $e_i\cap e_{i+1}=v_{i+1}$), and that at $v_i$, the smaller edge is $e_{i-1}$ and the bigger  is $e_i$. Then $\gamma_{v_i}$, as defined in Definition \ref{def:extopsecond}, is given by $e_{i}$. By Lemma \ref{lemma:extopsecond}, among the minimal shifts of external operators of the second kind, there is exactly one $x_i$ such that $f(x_i)(v_1)\neq 0$, namely the one corresponding to $T^{v_1}$ (since any other $\gamma_v$ would not contain $v_1$ as endpoint). Therefore $0=f(X)(v_1)=-2\lambda_i$ and thus $x_i$ appears with trivial coefficient in $X$. We can repeat the same reasoning in order for each $v_i$ to find that the corresponding minimal shift must not appear in $X$. Repeat this process for any other choice of introverted edges. Finally, take any $v$ in $\Phi_V$ that has at least one non introverted outgoing edge; then as before the only minimal shift such that $f(x_i)(v)\neq 0$ is the one corresponding to $T^{v}$ and thus following the same reasoning it must not appear in $X$.
\end{proof}
\begin{corollary}\label{cor:extoperators}
	Let $b_1,\dots,b_k$ be the external operators of $V$ (of both kinds); let $x_1,\dots,x_k$ be their unique minimal shift. Then $f(x_1),\dots,f(x_k)$ are linearly independent.
\end{corollary}
\begin{proof}
	Suppose the operators up to $b_n$ are of the first kind; then for any singular edge $e$, $f(x_i)(e)=0$ for all $i>n$. This implies that any linear dependence relation between the $f(x_i)$'s can only involve $i>n$ (since by Lemma \ref{lemma:exttopfirstindep} the first $n$ are linearly independent). However the $f(x_i)$'s for $i>n$ are linearly by Lemma \ref{lemma:exttopsecondindep}.
\end{proof}

\begin{definition}\label{def:intoperators}
	Consider $\gamma_1,\dots,\gamma_r$ all the maximal directed paths in $\Phi_V$; notice that they either end at a leaf of $\Phi_V$ or at an introverted edge. We will use these to define another set of operators but we first need to remove some of them. 
	
	Let $e_1,\dots,e_{n'}$ be the leaves of $\Phi_V$, and for each of them call $\gamma_{e_i}$ the maximal directed path ending in $e_i$ as in Lemma \ref{lemma:intopsecond}. 
	
	To these, we wish to add one maximal path ending at each introverted edge, but we need to make a choice. For each introverted edge $e$, there are exactly two maximal directed paths ending at $e$, each containing one of the endpoints $v,w$ of $e$ in their interior; we wish to remove one. Suppose that $\gamma_v$ contains $e$; then we remove the path that contains $v$ as an endpoint. If instead $\gamma_w$ contains $e$ we similarly remove the path that contains $w$ as an endpoint (notice that $\gamma_v$ and $\gamma_w$ cannot both contain $e$). If neither contains $e$, we simply remove one maximal path at random (see Figure \ref{fig:introoperator}). 
	
	Furthermore, for each extraverted edge $e$ of $\Phi_V$, we wish to remove one of these paths. To do this, choose $v$ the smaller endpoint of $e$, and call $\gamma_v$ the path defined in Lemma \ref{lemma:extopsecond}; one of the $\gamma_{e_i}$'s is going to equal $\gamma_v$ plus $e$, and we remove it from our collection (see Figure \ref{fig:extraoperator}) (notice that by construction this path was not removed by the previous step). Repeat the process for every extraverted edge.
	
	Renumber the remaining paths as $\gamma_{e_1},\dots,\gamma_{e_n}$, and define $d_{k+i}=T^{\gamma_{e_i}}$ as in Lemma \ref{lemma:intopsecond}.
	
	\begin{figure}
		\includegraphics[scale=0.4]{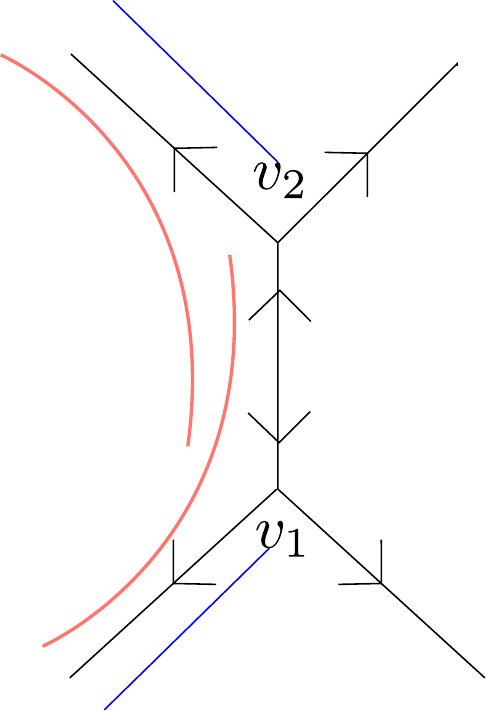}
		\label{fig:extraoperator}
		\caption{If $v_1\succ v_2$, the definition removes the top orange path from the collection.}
	\end{figure}
	
	\begin{figure}
		\includegraphics[scale=0.7]{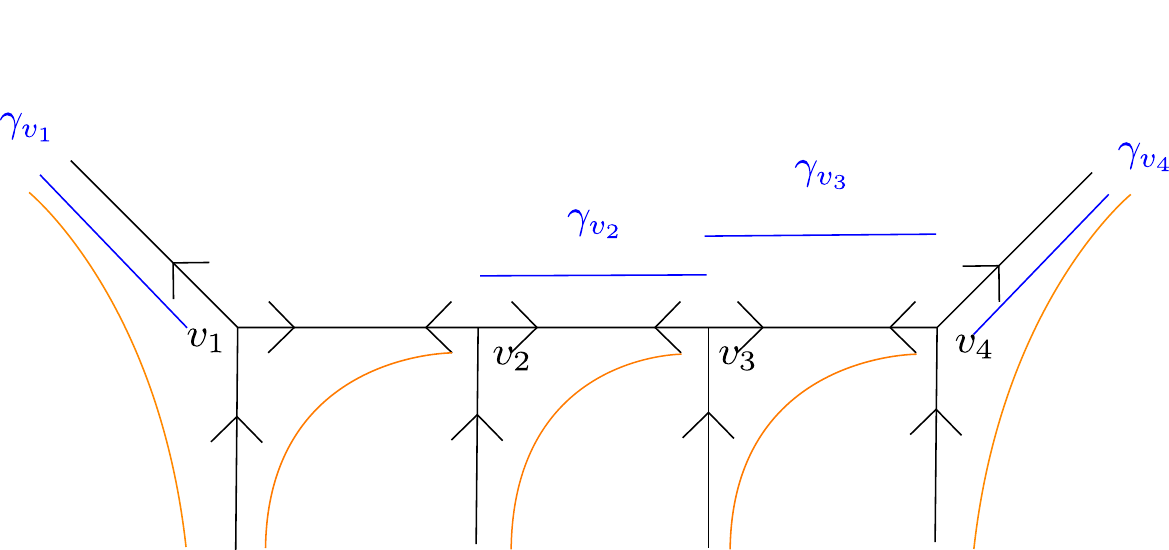}
		\label{fig:introoperator}
		\caption{An example of set of paths defining the operators. The blue straight paths define the external operators, while the orange curved paths define the internal operators.}
	\end{figure}
	
	Finally, let $e_1,\dots,e_m$ be the regular edges of $\Gamma$. Define $d_{k+n+i}:=T^{e_i}$ as in Lemma \ref{lemma:intopfirst}, so that we get operators $d_{k+1},\dots,d_{k+n+m}$, the first $n$ each corresponding to a maximal directed paths in $\Phi_V$, and the last $m$ corresponding to regular edges. Call these operators the \emph{internal operators} of $V$.
\end{definition}

\begin{lemma}\label{lemma:basisV}
	Let $d_{k+1},\dots,d_{k+n+m}$ be the internal operators of $V$, and choose a non-zero minimal shift for each of them; call them $y_{k+1},\dots,y_{k+n+m}$. Then $g(y_{k+1}),\dots,g(y_{k+n+m})$ are a basis for $g(V)$ over $\Q$ (hence $k+n+m=3g-3$).
\end{lemma}
\begin{proof}
	The fact that they are linearly independent is a direct and obvious consequence of Lemmas \ref{lemma:intopfirst}, \ref{lemma:intopsecond} and \ref{lemma:intopsecondintro}.
	
	We need to prove that $g(W)$ is spanned over $\Q$ by $g(y_{k+1}),\dots,g(y_{k+n+m})$, where $W$ is the smallest subspace of $\Z^\mathcal{E}$ that contains $V$. To see this, we just need to show that given $a_{k+1},\dots,a_{k+n+m}$ arbitrarily, there exists a $c\in W$ such that $c(e_i)=a_i$. We do so by using the equations defining $W$ to extend $a_{k+1},\dots,a_{k+n+m}$ to a coloring of $\Gamma$. Every regular or singular edge has its color given by either the $a_i$'s or the definition of $W$. The edges of $\Phi_V$ can be colored from the leaves and introverted edges inwards; notice that if $\Phi_V$ has an extraverted edge $e$, we removed one of the leaves or introverted edges reachable from $e$, therefore we can first determine the color of $e$ by working inwards on the other side, and then use $e$ to determine the color of the missing edge.
\end{proof}

\begin{remark}
	Notice that if $d_{i}$ corresponds to a regular edge, it has $2$ different non-zero minimal shifts with opposite signs, and thus its polytope of minimal shifts is given by a segment that contains $0$.
\end{remark}

The following lemma is the key reason for choosing the operators the way we did; we will use it to guarantee that in the recurrence relations we obtain from them, there are no cancellations among the minimal shifts.
\begin{lemma}\label{lemma:fglinindep}
	Let $b_1,\dots,b_k,d_{k+1},\dots,d_{3g-3}$ be the operators from Definitions \ref{def:extopfirst}, \ref{def:extopsecond} and \ref{def:intoperators}, and $y_1,\dots,y_{3g-3}$ be a non-zero minimal shift for each of them. Then the vectors $\xi_i:=(f(y_i),g(y_i))\in\Z^{N}$ are linearly independent.
\end{lemma}

\begin{proof}
	Suppose that $\Xi:=\sum_{i=1}^{3g-3} a_i \xi_i=0$. 
	
	First notice that if $e$ is a regular edge, and $d_j$ is the corresponding internal operator, then $f(y_j)=0$ and $g(y_j)=\pm2\delta_e$, and all other $y_i$'s have disjoint support from $\delta_e$; therefore $a_j=0$. Similarly, if $e$ is a singular edge, $\xi_i(e)=0$ unless $\xi_i$ is the minimal shift of an external operator of the first kind. Therefore by evaluating $\Xi$ on the subspace spanned by singular edges, we obtain that (by Lemma \ref{lemma:exttopfirstindep}) all those $\xi_i$'s must have trivial coefficient in $\Xi$.
	
	Therefore we only need to worry about external operators of the second kind and internal operators associated to a maximal directed path in $\Phi_V$.
	
	Now we introduce some notation that will greatly simplify the flow of the proof. For any $v\in \Phi_V$, we have the path $\gamma_v$ defined in Definition \ref{def:extopsecond}, the associated path operator, and the unique minimal shift which we will denote $y_v$. With a small abuse of notation, we will say that $y_v$ is the minimal shift of $\gamma_v$. Similarly, any maximal path in $\Phi_V$ among the ones defining the internal operators is determined by its final edge $e$; we call this path $\gamma_e$ and the minimal shift of the resulting operator $y_e$. Notice that $e$ must be either a leaf of $\Phi_V$ or an introverted edge. For any of these minimal shifts, call $\xi_v:=\left(f(y_v),g(y_v)\right)$ (and similarly for $\xi_e$). Therefore we need to consider $\Xi=\sum_e a_e \xi_e+\sum_v a_v\xi_v=0$; we say that $v$ or $e$ are \emph{in the support of $\Xi$} if its coefficient is non-zero in $\Xi$. 
	
	We break up the proof in several steps.

	\emph{Step 1: if the incoming edge to $v$ is directed, and no outgoing edges is introverted, then $v$ does not belong to the support of $\Xi$.}
	
	Based on the assumption, $v$ does not belong to a root or an extraverted edge; it also does not belong to any introverted edge. Then Lemmas \ref{lemma:intopsecond} and \ref{lemma:intopsecondintro} imply that for any $e$, $\xi_e(v)=0$; similarly Lemma \ref{lemma:extopsecond} implies that $\xi_{v'}(v)=0$ for $v\neq v'$. Therefore $0=\Xi(v)=-2a_v$ implies $a_v=0$.

	\emph{Step 2: if the incoming edge to $v$ is directed, and exactly one outgoing edge is introverted, then $v$ is not in the support of $\Xi$.}
		
	Suppose $e'$ is the introverted outgoing edge and $e$ is the other one (either a leaf or directed); then $\gamma_v=e$; call $w$ the other endpoint of $e'$. 
	
	If $\gamma_w$ does not contain $e'$, then no other $\gamma_{w'}$ does, which means that $0=\Xi(e')=-2a_{e'}=0$ and $0=\Xi(v)=-2a_v=0$ and we have the desired result. 
	
	Suppose instead that $\gamma_w=e'$: then by definition $\gamma_{e'}$ contains $w$ in its interior and $0=\Xi(e')=-2a_{e'}-2\sum a_{w'}$ where the sum runs over all $w'$s such that $\gamma_{w'}$ contains $e'$.  By evaluating $\Xi$ in $v$ we then obtain that $0=\Xi(v)=-2a_v-2a_{e'}-2\sum a_{w'}=-2a_v$ and $v$ is not in the support of $\Xi$.

	\emph{Step 3: if the incoming edge to $v$ is directed, and both outgoing edges are introverted, then $v$ is not in the support of $\Xi$.}
	
	Call $e,e'$ the outgoing edges, and suppose that $e$ is larger than $e'$, so that $\gamma_v=e$. Then $\gamma_e$ has $v$ in its interior; suppose that $\gamma_{e'}$ contains $v$ as an endpoint. Call $W_{e'}$ the set of vertices such that $e'\subseteq \gamma_w$. For all $w\in W_{e'}$, $\gamma_{w}$ has $v$ as its endpoint, which means that $0=\Xi(v)=-2a_v-2a_{e'}-2\sum_{w\in W_{e'}}a_w$. From $0=\Xi(e')=-2a_{e'}-2\sum_{w\in W_{e'}}$ we find that $v$ does not belong to the support of $\Xi$. 
	
	If instead $\gamma_{e'}$ does not contain $v$ as an endpoint, then by Definition \ref{def:intoperators} $\gamma_w$ does not contain $v$, where $w$ is the other endpoint of $e'$. In this case, no $w'\neq v$ is such that $v$ is an endpoint of $\gamma_{w'}$ and $v$ is not in the support of $\Xi$ as before.
	
	To recap, at this step the only vertices that could belong to the support of $\Xi$ belong to roots or extraverted edges.
	
	\emph{Step 4: if $v$ is the larger endpoint of an extraverted edge, then $v$ does not belong to the support of $\Xi$.}
	
	Call $e$ the final edge of $\gamma_v$. When defining the operators we did not use $\gamma_e$; furthermore, any $w\neq v$ such that $e\subseteq\gamma_w$ is not in the support of $\Xi$. This is because if $\gamma_w\cap \gamma_v$ is non-empty then $\gamma_w\subseteq \gamma_v$ and thus $w$ has a directed incoming edge (therefore is not in the support by the previous steps). This means that $0=\Xi(e)=-2a_v$ and we have the desired result.
	
	\emph{Step 5: if $e$ is a leaf or an introverted edge, and $\gamma_e$ contains the larger endpoint of an extraverted edge in its interior, then $e$ does not belong to the support of $\Xi$.}
	
	By the previous steps there are no $w$'s in the support of $\Xi$ such that $e\subseteq \gamma_w$, which means that $0=\Xi(e)=-2a_e$ and $e$ does not belong to the support of $\Xi$.
	
	\emph{Step 6: no $v$ belongs to the support of $\Xi$.}
	
	The only cases left to check are when $v$ belongs to a root of $\Phi_V$, or when it is the smaller endpoint of an extraverted edge. We carry out the latter since the former will be identical.
	
	Let $e$ be the final edge of $\gamma_v$. If $e'\neq e$ and $v$ is in the interior of  $\gamma_{e'}$, then $e'$ is not in the support of $\Xi$. To see this, notice that there are no $w$'s in the support of $\Xi$ such that $e'\subseteq \gamma_w$ (following the same reasoning as Step 4), which means that $0=\Xi(e')=-2a_{e'}$. Therefore $0=\Xi(v)=-2a_v-4a_e$. 
	
	Then no $w\neq v$ such that $e\subseteq\gamma_w$ belongs to the support of $\Xi$. As usual if $e'\neq e$, $e$ is not contained in $\gamma_{e'}$. This means that $0=\Xi(e)=-2a_e-2a_v$, which combined with $-2a_v-4a_e=0$ implies that $v$ is not in the support of $\Xi$.
	
	\emph{Step 7: no $e$ belongs to the support of $\Xi$.}
	
	Because there are no $v$'s in the support of $\Xi$, and because $y_{e}(e')=0$ if $e\neq e'$, we have that $0=\Xi(e)=-2a_e$ and $e$ is not in the support of $\Xi$.
	
\end{proof}
\begin{corollary}\label{cor:linindep}
The elements $y_1,\dots,y_{3g-3}$ of Lemma \ref{lemma:fglinindep} are linearly independent.
\end{corollary}
\begin{lemma}\label{lemma:polytope}
	Suppose that among the internal operators $d_{k+1},\dots,d_{3g-3}$, the last $l$ correspond to the regular edges $e_{3g-3-l+1},\dots,e_{3g-3}$; denote with $x_j$ the unique minimal shift of the $j$-th operator for $j\leq 3g-3-l$. Consider the monomial $M=\prod_{i=1}^{k}b_i^{n_i}\prod_{j=k+1}^{3g-3}d_j^{n_j}$. Then the polytope of minimal shifts of $M$ is given by $$\sum_{i=1}^{3g-3-l}n_ix_i+\left\{\sum_{j=3g-2-l}^{3g-3}r_j\delta_{e_j},-2n_j\leq r_j\leq 2n_j\right\}$$.
	
\end{lemma}

\begin{proof}
	First notice that the operator $M'=\prod_{i=1}^{k}b_i^{n_i}\prod_{j=k+1}^{3g-3-l}d_j^{n_j}$ has a unique minimal shift equal to $x'=\sum_{i=1}^{3g-3-l}n_ix_i$; this can be proved by repeated applications of Lemma \ref{lemma:uniqueminimalshift}. Following this, notice that applying Lemma \ref{lemma:curveActions} repeatedly we can see that the polytope of minimal shifts of $d_{j}^{n_j}$ is a segment whose endpoints are $\pm2n_j\delta_{e_j}$; once again applying Lemma \ref{lemma:uniqueminimalshift} we obtain that the polytope of minimal shifts of $M'':=\prod_{j=3g-2-l}^{3g-3}d_j^{n_j}$ is $\left\{\sum_{j=3g-2-l}^{3g-3}r_j\delta_{e_j},-2n_j\leq r_j\leq 2n_j\right\}$. A final application of Lemma \ref{lemma:uniqueminimalshift} gives the desired result.
\end{proof}

\begin{corollary}\label{cor:polytope}
	Take $M_1,M_2$ two monomials as in Lemma \ref{lemma:polytope}, and call $\mathcal{P}_1,\mathcal{P}_2$ their respective polytopes of minimal shifts. Then:
	\begin{itemize}
		\item if $\mathcal{P}_1\cap \mathcal{P}_2\neq\emptyset$, then $M_1$ and $M_2$ have the same multidegree up to variable $d_{3g-3-l}$ (i.e. except for variables corresponding to regular edges);
		\item if $\mathcal{P}_1=\mathcal{P}_2$ then $M_1=M_2$.
		
	\end{itemize}
\end{corollary}
\begin{proof}
	Call $(n_1,\dots,n_{3g-3})$ the multidegree for $M_1$, and $(s_1,\dots,s_{3g-3})$ the multidegree for $M_2$.
	Suppose that there are $y_1,y_2$ belonging to $\mathcal{P}_1,\mathcal{P}_2$ respectively and that $y_1=y_2$. We know from Lemma \ref{lemma:fglinindep} that the $x_i$'s and the $\delta_{e_j}$'s are linearly independent, which implies that $n_i=s_i$ for all $i\leq 3g-3-l$, proving the first part of the lemma. If furthermore $\mathcal{P}_1=\mathcal{P}_2$ then we can repeat the same argument for every vertex of this polytope to prove that the rest of the multidegrees must coincide too.
\end{proof}

\begin{lemma}\label{lemma:effectivemonomial}
	Let $V$ be a simple barely admissible subspace, and suppose $M_1,M_2$ are operators with the following property: there exist $U_1,U_2$, each a finite union of barely admissible subspaces, such that
	\begin{itemize}
		\item for any $c\in \Delta\setminus U_i$, $c$ is effective for $M_i$;
		\item $U_i$ is transverse to $V$.
	\end{itemize}
	
	 Then there exists $U\subsetneq \Delta$ finite union of barely admissible subspaces such that for any $c\in \Delta\setminus U$, $c$ is effective for $M_1M_2$ and $U$ is transverse to $V$.
\end{lemma}
\begin{proof}
	Let $x_1,\dots,x_n$ be the minimal shifts of $M_1$ and $y_1,\dots,y_r$ the shifts of $M_2$, with $y_1$ minimal. We show that $U$ can be chosen as the union of all subspaces $U_2$ and $y_j-y_1+U_1$. If $c$ is effective for $M_2$, and $c+y_j-y_1$ is effective for $M_1$ for all $j$, then $c$ is effective for $M_1M_2$, by simply applying $M_1$ to each summand in $M_2\cdot \phi_{c+y_1}$.
	
\end{proof}

\begin{corollary}\label{cor:effectivemonomial}
	Let $W$ be a subspace of $\Z^{\mathcal{E}}$, and let $V$ be the smallest simple barely admissible subspace containing $W$. Let $M$ be a monomial in the operators associated to $V$; then there exists a subset $U\subsetneq W$, consisting of a finite union of barely admissible subspaces, such that for any $c\in W\setminus U$, $c$ is effective for $M$.
\end{corollary}
\begin{proof}
	The case of $W=V$ is an immediate consequence of Lemmas \ref{lemma:extopfirst}, \ref{lemma:extopsecond}, \ref{lemma:intopfirst}, \ref{lemma:intopsecond}, \ref{lemma:intopsecondintro} establishing the result for each variable, and \ref{lemma:effectivemonomial} establishing the result for a product.
 The general case follows at once, since $W$ cannot be contained in $U$ by minimality of $V$.
\end{proof}

\subsection{Proof of the inductive case}\label{sec:induction2}

\begin{proof}[Proof of the inductive step in Case 2]
	The proof follows the same general lines of Case 1. We set $n=3g-3$ for convenience. Recall that we have a collection $(V_i)_{i\in I}$ of subspaces of $\Z^{\mathcal{E}}=\Z^{n}$, each of codimension at least $1$, such that $$\mathrm{Span}_{\Q(A,\overline{\lambda})}\lbrace \varphi_c \ | \ c\in \Delta \setminus \underset{s\in S}{\bigcup}V_s\rbrace \ \textrm{is finite dimensional}$$
	
	and that we want to prove that we can find $i\in I$, and $V_i^j\subsetneq V_i$ for $j\in S$ such that the collection $(V_s)_{s\in I\setminus i}\cup (V_i^j)_{j\in S}$ also satisfies the above condition.
	
	To start, pick $V_0$ a subspace of minimal codimension among the $V_i$'s; call $\widetilde{V_0}$ the smallest simple barely admissible subspace that contains $V_0$. Recall that $\widetilde{V_0}$ comes with maps $f,g$, with a regular, singular and flow subgraphs, with the internal and external operators of Lemma \ref{lemma:fglinindep}, and with an order on the edges and vertices of the flow subgraph.
	
	 Among the $V_i$'s there might be parallels of $V_0$; our first order of business is to pick correctly one of these subspaces. For any $V_i$ parallel to $V_0$, consider $f(\widetilde{V_i})$ where $\widetilde{V_i}$ is the smallest simple barely admissible subset containing $V_i$ (notice that clearly $\widetilde{V_i}$ and $\widetilde{V_0}$ are parallel). Replace $V_0$ with one of the parallel subspaces that maximize $f(\widetilde{V_i})$ (notice that there could be more than one). Now let $e_1,\dots,e_{r}$ the regular edges and the final edges of internal operators of $\widetilde{V_0}$. Then $\delta_{e_1},\dots,\delta_{e_r}$ form a basis for $g(\widetilde{V_0})$ (in particular $r=\widetilde{k}$). We pick a subset $J\subseteq \{1,\dots,\widetilde{k}\}$ applying Lemma \ref{lemma:lin_alg} to $g(V_0)\subseteq g(\widetilde{V_0})$. Any translate of $V_0$ contained in $\widetilde{V_0}$ will be of the form $V_0+\sum_{j\in J} x_j\delta_{e_j}$. Among these subspaces that belong to $(V_i)_{i\in I}$, we redefine $V_0$ to be the one maximizing $\left(g(x_j)\right)_{j\in J}$ in the lexicographic order.
	
	Finally pick $J'\subseteq \{1,\dots,n\}$ according to Lemma \ref{lemma:non-vanishing_poly}; let $a_i:=T^{\alpha_i}$. Let $b_1,\dots,b_{n-\widetilde{k}}$ be the external operators of $\widetilde{V_0}$. Let $d_1,\dots,d_{\widetilde{k}}$ be the internal operators of $\widetilde{V_0}$, numbered so that $d_i$ corresponds to the edge $e_i$.

	 For any $l\in \{1,\dots,\widetilde{k}\}\setminus J$ we consider the family of operators
	$$F_l:=\{a_{j'}\}_{j'\in J'}\cup \{d_j\}_{j\in J\cup \{l\}}\cup\{b_1,\dots,b_{n-\widetilde{k}}\}.$$
	
	For any $l,$ the set $F_l$ contains exactly $3g-2$ multicurve operators, therefore by Lemma \ref{lem:linDep} there is a non-zero polynomial $P_l$ such that $P_l(a_{j'},d_j,b_i)\cdot\varphi_c=0$ as an element of $\Sk^{\overline{\lambda}}(M)$ for any $c\in \Delta$.
	
	Let $D_l$ be the total degree of $P_l$ in the variables $b_i$ and $d_j$; consider $W\subseteq V_0$ the subset with the following properties:
	\begin{itemize}
		\item $c\in W$ if $c$ is not effective for at least one monomial of $P_l$; or 
		\item $c\in W$ if $c$ is at distance less than or equal to $4D_l$ from any other subspace $V_i$, with $i\in I$, such that $V_i$ is not parallel to $V_0$.
	\end{itemize}
	
	By Corollary \ref{cor:effectivemonomial}, $W$ is a finite union of subspaces of $V_0$ of codimension at least $1$. Call $W_0:=V\setminus W$.
	
	\textbf{Claim:} For any $l\in \{1,\dots,k\}\setminus J$ and $c\in W_0$, we have a relation of the form
	
	$$ R_{d_l}(A,A^{c_{i_1}},\ldots , A^{c_{i_k}})\varphi_{c+d_l u_l} + \ldots + R_{-d_l}(A,A^{c_{i_1}},\ldots , A^{c_{i_k}})\varphi_{c-d_l u_l} \equiv 0,$$
	where the vector $u_l$ is a generator of the line $D_l$ given in Lemma \ref{lemma:lin_alg}, $d_l$ is an integer, $R_{-d_l}$ is a non-zero polynomial in $k+1$ variables, and $\equiv$ is equality up to an element of 
	$$E_0=\mathrm{Span}_{\Q(A,\overline{\lambda})}\lbrace \varphi_c \ | \ c \in \Delta \setminus \underset{s\in S}{\bigcup} V_s \rbrace.$$
	
	\begin{proof}[Proof of the Claim]
		We assume for simplicity that no final edges of the internal operators are loop edges. If any final edge is a loop edge, it simply changes some of the following formulas by removing a factor of $2$ in some minimal shift (since the minimal shift of the curve operator of a loop edge is $-\delta_e$ rather than $-2\delta_e$). This is inconsequential and only makes formulas more awkward.
		
		Consider $M$ a monomial of $P_l$. The $a_i$'s act on $\varphi_c$ as multiplication by a scalar, so $M\cdot \varphi_c=R(A,A^{c_{i_1}},\dots,A^{c_{i_k}}) M'\cdot \varphi_c$ where $M'$ is simply $M$ with all $a_i$ variables removed. Rewrite all monomials of $P_l$ like this, and call $M_1,\dots,M_r$ the resulting monomials. Let $\mathcal{P}_1,\dots,\mathcal{P}_r$ their polytopes of minimal shifts; by Lemma \ref{lemma:polytope}, $f(\mathcal{P}_i)=f(x_i)$ for some $x_i$; select a monomial $M_i$ following these three criteria, in order of importance:
		\begin{itemize}
			\item $f(x_i)$ minimal;
			\item $g(x_i)$ minimal;
			\item maximal degree in $d_l$;
			\item maximal degree in the other $d_j$s, in the lexicographic order.
		\end{itemize}
		
		Let $n_j$ be the degree of $M_i$ in $d_j$. 
		
		By Lemma \ref{lemma:polytope} and Corollary \ref{cor:polytope} there is no other monomial $M_r$ of $P_r$ that contains the shifts $x_\pm:=x_i \pm 2n_l\delta_{e_l}-\sum_{j\neq l} 2n_j\delta_{e_j}$, where the sum is over regular edges (notice that if $e_l$ is not regular, $x_+$ is not a minimal shift). For any $c\in W_0$, let $\widetilde{c}:=c-x_i+\sum_{j\neq l} 2n_j\delta_{e_j}$.
		
		For any monomial $M_r$, if $x_r\neq x_i$, $M_r\cdot \varphi_{\widetilde{c}}\in E_0$. To see this, suppose that there is a shift $x$ such that $\varphi_{\widetilde{c}+x}\in \underset{s\in S}{\bigcup} V_s$. By construction $\widetilde{c}+x$ cannot belong to one of the $V_s$ that are not parallel to $V_0$, since the distance of $c$ from any of these is more than $4D_l$, and both $x$ and $c-\widetilde{c}$ have $L^1$-norm less than or equal to $2D_l$. Therefore $V_s$ must be parallel to $V_0$; however in this case because of our choice of $V_0$, we would have either $$f(c)\geq f(\widetilde{c}+x)=f(c)+f(x)-f(x_i)\geq f(c)+f(x_r)-f(x_i)>f(c)$$
		by minimality (if $f(x)>f(x_i)$), or if $f(x)=f(x_i)$, then $f(x)$ is minimal and we would have $$g(c)\geq g(\widetilde{c}+x)=g(c)+g(x)-g(x_i)\geq g(c)+g(x_r)-g(x_i)>g(c).$$
		
		If instead $x_r=x_i$, $$M_r\cdot \varphi_{\widetilde{c}}=R_{n'_l}\varphi_{c+2n'_l \delta_{e_l}}+R_{-n'_l}\varphi_{c-2n'_l \delta_{e_l}}+\sum_{j=-n'_l+1}^{n'_l-1}R_{j}\varphi_{c+2j\delta_{e_l}}+X$$
		
		with $X\in E_0$ and $R_{-n_r}\neq 0$, where $n'_j$ is the degree of $M_r$ in the variable indexed by $j$. To see this, notice that if $x$ is a non-minimal shift, then we can apply the above reasoning to show that $\varphi_{\widetilde{c}+x}\in E_0$. Therefore the only summands in $M_r\cdot\varphi_{\widetilde{c}}$ are the ones corresponding to the minimal shifts of $M_r$. Take one such minimal shift $x:=x_i+\sum_{j} \pm2m_j\delta_{e_j}$ with $m_j\leq n'_j$ the degree of $M_r$ in the variable $d_j$. If $m_j<n_j$ for any $j\neq l$, then $\widetilde{c}+x\in E_0$, using the same reasoning as before. Therefore the only minimal shifts for which $\varphi_{\widetilde{c}+x}\in V_0$ are $$x_l:=x_i \pm 2m_l\delta_{e_l}-\sum_{j\neq l} 2n_j\delta_{e_j}$$
		for $-n_l\leq m_l\leq n_l$, which gives the desired result.

		Combining these two facts, we obtain that 		
		$$P_l\cdot \varphi_{\widetilde{c}}=R_{n_l}\varphi_{c+2n_l \delta_{e_l}}+R_{-n_l}\varphi_{c-2n_l \delta_{e_l}}+\sum_{j=-n_l+1}^{n_l-1}R_{j}\varphi_{c+2j\delta_{e_j}}+X$$
		
		where $R_{- n_l}$ is a non-zero Laurent polynomial evaluated in $A,A^{c_1},\dots,A^{c_n}$ as before, and $X\in E_0$.
		
	\end{proof}
	
	From here, the proof of the inductive case proceeds verbatim as in Case 1.
\end{proof}
\iffalse

\fi
\section{Proof of corollaries}
\label{sec:corollaries}
In this section, we prove the various corollaries mentioned in the introduction.

First, we tackle the proof of Corollary \ref{cor:explicit-genset}, that says that our proof provides an algorithmic way of constructing a generating of $Sk^{\overline{\lambda}}(M)$ for any $3$-manifold $M.$

\begin{proof}[Proof of Corollary \ref{cor:explicit-genset}]
	
	We need to argue that every step along the course of the proof of Theorem \ref{thm:main} was algorithmic.
	
	We start with the data of a Heegaard diagram of a closed oriented $3$-manifold.
	
	Over the course of the proof we use Lemma \ref{lem:linDep} repeatedly  to get vanishing polynomials for sets of $3g-2$ curves on the Heegaard surface.
	
	Lemma \ref{lem:linDep} can be made algorithmic: indeed, one can resolve all crossing in monomials $c_1^{n_1}\ldots c_{3g-2}^{n_{3g-2}}$ to express them in terms of multicurves on the surface, which form a basis of $Sk(\Sigma).$ The fact that monomials will eventually be linearly dependent will be apparent on their coefficients on the basis of multicurves, and we will be able to compute a vanishing polynomial.
	
	Fusing rules clearly provide an explicit way to turn such vanishing polynomials into recurrence relations among the basis vectors $\varphi_c$ of $Sk(H)$ in $Sk(M).$
	
	The other preliminary lemmas, Lemma \ref{lemma:non-vanishing_poly} and \ref{lemma:lin_alg} are based on linear algebras and clearly algorithmic. Lemma \ref{lemma:chamber} also provides an explicit way of extracting a finite set of generators in each chamber.
	
	Finally, in the main body of the proof of Theorem \ref{thm:main}, the choice of subspaces, of internal and external curve operators, etcetera, is always explicit.
\end{proof}

Our next proof deals with Corollary \ref{cor:special-values}, which we recall below:
\begin{corollary}[Corollary \ref{cor:special-values}]
	Let $M$ be a closed $3$-manifold. Then there exists an integer $n\geq 1$ and a polynomial $R\in \Z[A^{\pm 1}][X_1,\ldots, X_n],$ such that, for any $\zeta\in \C$ which is not a root of any non-zero polynomial of the form $R(A^{k_1},\ldots,A^{k_n})$ with $k_i \in \Z,$ we have
	$$\dim_{\C}Sk_{\zeta}(M) <+\infty.$$
\end{corollary}
\begin{proof}[Proof of Corollary \ref{cor:special-values}]
	Let $M$ be any $3$-manifold as in Theorem \ref{thm:main}. Perform the algorithm outlined by the proof of Theorem \ref{thm:main}; along the way, we used a finite number of recurrence relations and we required that their leading coefficients, Laurent polynomials of the form $R_i(A,A^{c_1},\dots,A^{c_n})$, did not vanish. Define $R$ as the product of all such $R_i$'s. Then if $\zeta$ is not a root of $R(A,A^{c_1},\dots,A^{c_n})$, we can repeat the proof verbatim for $Sk_\zeta^{\overline{\lambda}}$, with the guarantee that the leading coefficients of our recurrence relations do not vanish.
\end{proof}
As mentioned in the introduction, Corollary \ref{cor:special-values} implies that for any closed $3$-manifold $M,$ there are algebraic numbers $\zeta$ such that $\dim S_{\zeta}(M)<+\infty.$ Indeed, we have:
\begin{lemma}\label{lemma:special-roots} Let $n\geq 1$ and let $R\in \Z[A^{\pm 1}][X_1,\ldots,X_n],$ and let $\mathcal{F}\subset \overline{\Q}$ be the set of roots of all non-zero polynomials of the form $R(A,A^{c_1},\ldots,A^{c_n}).$ Then 
	\begin{itemize}
		\item[(1)] There exists $C>0,$ such that for any $\zeta$ with $|\zeta|>C$ then $\zeta\notin \mathcal{F}.$
		\item[(2)] There exists $N\geq 1,$ such that for any $\zeta \in \overline{\Q}$ with at least $N$ real conjugates, $\zeta\notin \mathcal{F}.$
	\end{itemize}
\end{lemma}
We note that condition (2) and its proof was suggested to us in a MathOverflow discussion \cite{MathOver}.
\begin{proof}
(1) Let $B\in \Z$ the sum of all coefficients of all monomials in $R.$ For any $c_1,\ldots,c_n\in \Z,$ the polynomial $R(A,A^{c_1},\ldots,A^{c_n})\in \Z[A^{\pm 1}]$ has all its coefficients bounded above by $B$ in absolute value. However, by Cauchy's bound, any root $\alpha$ of a polynomial $P\in \C[A^{\pm 1}]$ with dominant coefficient $c_d$ and other coefficients bounded by $M$ in moduli are such that $|\alpha|\leq 1+\frac{M}{c_0}.$ Therefore, we have that $|\alpha|\leq B+1$ for any $\alpha \in \mathcal{F}.$

(2) Let $n$ be the number of monomials in $R.$ Then for any $c_1,\ldots,c_n\in \Z,$ the polynomial $R(A,A^{c_1},\ldots,A^{c_n})\in \Z[A^{\pm 1}]$ has at most $n$ non-zero coefficients. However, a polynomial with $n$ non-zero coefficients has at most $2n-2$ real roots. Therefore, if $\zeta \in \overline{\Q}$ has at least $N=2n-1$ real conjugates, then $\zeta \notin \mathcal{F}.$
\end{proof}
Finally, we prove Corollary \ref{cor:peripheral-ideal}, which we recall here:
\begin{corollary}[Corollary \ref{cor:peripheral-ideal}]
	Let $M$ be a compact oriented $3$-manifold with boundary not a disjoint union of spheres, and let $x\in Sk(M,\Z[A^{\pm 1}]).$ Then there exists $z\in Sk(\partial M,\Z[A^{\pm 1}])$ such that $z\neq 0$ and $z\cdot x=0.$
\end{corollary}
We start by the following lemma. For $x_1,\ldots,x_k$ elements of a non-commutative algebra $\mathcal{A}$, an ordered polynomial will be a linear combination of ordered monomials $x_1^{n_1}\ldots x_k^{n_k}.$
\begin{lemma}
	\label{lemma:ordered-polynomial} Let $\Sigma$ be a surface, and $\lambda_1,\ldots,\lambda_k$ and $\delta$ be simple closed curves on $\Sigma,$ with $\lambda_1,\ldots, \lambda_k$ pairwise disjoint, $\delta$ disjoint from all curves $\lambda_i$ except $\lambda_1,$ and $\delta$ and $\lambda_1$ intersecting either once or twice with opposite orientations.
	
	Then, for any ordered polynomial $P$ with non-zero coefficients in $\lambda_1,\ldots,\lambda_k,\delta,$ one has that $P\neq 0 \in Sk(\Sigma,\Z[A^{\pm 1}]).$
\end{lemma}
\begin{proof}
	Because $\lambda_2,\ldots,\lambda_k$ commute with all other curves involved, it is sufficient to consider the case $k=1,$ in which case $\delta$ and $\lambda_1$ can be viewed as curves on a one-holed torus or a $4$-holed sphere, and the claim is clear, for instance from the presentation of the skein algebras of those surfaces given in \cite{BP00}.
\end{proof}

We are now ready for the proof of Corollary \ref{cor:peripheral-ideal}:

\begin{proof}[Proof of Corollary \ref{cor:peripheral-ideal}]
	Take $\overline{\lambda}$ to be a pants decomposition of $\partial M,$ and let $\delta$ be a non-trivial simple closed curve on $\partial M$ such that $(\overline{\lambda},\delta)$ satisfies the hypothesis of Lemma \ref{lemma:ordered-polynomial}.
	Then the elements $\delta^n \cdot x$ are linearly dependent in $Sk^{\overline{\lambda}}(M),$ thus
	$$\underset{i_1,\ldots i_{3g-3}\geq 0, n\geq 0}{\sum} c_{i,n} \lambda_1^{i_1}\ldots \lambda_{3g-3}^{i_{3g-3}} \delta^n \cdot x =0 \in Sk^{\overline{\lambda}}(M)$$
	for some coefficients $c_{i,n}\in \Z[A^{\pm 1}]$ that are not all zero. However, an element of $Sk(M)$ is zero in $Sk^{\overline{\lambda}}(M)$ if and only if it is $P(A,\lambda_1,\ldots,\lambda_n)$-torsion for some polynomial $P.$ Therefore, up to multiplying by $P,$ we have the same equality in $Sk(M,\Z[A^{\pm 1}]).$ However, any non-zero (ordered) polynomial in $\lambda_1,\ldots,\lambda_n,\delta$ is non-zero in $Sk(\partial M,\Z[A^{\pm 1}])$ by Lemma \ref{lemma:ordered-polynomial}.
\end{proof}
\bibliographystyle{hamsalpha}
\bibliography{biblio}

\providecommand{\bysame}{\leavevmode\hbox to3em{\hrulefill}\thinspace}
\providecommand{\href}[2]{#2}
\providecommand{\eprint}{\begingroup \urlstyle{rm}\Url}
\begin{thebibliography}{BKSW25}

\bibitem[AF22]{AF22}
Jos\'e{}~Rom\'an Aranda and Nathaniel Ferguson, \emph{Generating sets for the
  {K}auffman skein module of a family of {S}eifert fibered spaces}, New York J.
  Math. \textbf{28} (2022), 44--68.

\bibitem[BKSW25]{BKSW}
Rhea~Palak Bakshi, Seongjeong Kim, Shangjun Shi, and Xiao Wang, \emph{On the
  kauffman bracket skein module of $(s^1 \times s^2) \ \# \ (s^1 \times s^2)$},
  2025, \eprint{2405.04337}.

\bibitem[BP00]{BP00}
Doug Bullock and J\'ozef~H. Przytycki, \emph{Multiplicative structure of
  {K}auffman bracket skein module quantizations}, Proc. Amer. Math. Soc.
  \textbf{128} (2000), no.~3, 923--931.

\bibitem[Car17]{Car}
Alessio Carrega, \emph{Nine generators of the skein space of the 3-torus},
  Algebr. Geom. Topol. \textbf{17} (2017), no.~6, 3449--3460.

\bibitem[Det21]{Det21}
Renaud Detcherry, \emph{Infinite families of hyperbolic 3-manifolds with
  finite-dimensional skein modules}, J. Lond. Math. Soc. (2) \textbf{103}
  (2021), no.~4, 1363--1376.

\bibitem[DKS25]{DKS}
Renaud Detcherry, Efstratia Kalfagianni, and Adam~S. Sikora, \emph{Kauffman
  bracket skein modules of small 3-manifolds}, Adv. Math. \textbf{467} (2025),
  Paper No. 110169, 45.

\bibitem[DS25]{DS25}
Renaud Detcherry and Ramanujan Santharoubane, \emph{An embedding of skein
  algebras of surfaces into localized quantum tori from {D}ehn-{T}hurston
  coordinates}, Geom. Topol. \textbf{29} (2025), no.~1, 313--348.

\bibitem[EJ70]{EJ70}
Jack Edmonds and Ellis~L. Johnson, \emph{Matching: {A} well-solved class of
  integer linear programs}, Combinatorial {S}tructures and their {A}pplications
  ({P}roc. {C}algary {I}nternat. {C}onf., {C}algary, {A}lta., 1969), Gordon and
  Breach, New York-London-Paris, 1970, pp.~89--92.

\bibitem[FGL02]{FGL02}
Charles Frohman, R\u{a}zvan Gelca, and Walter Lofaro, \emph{The {A}-polynomial
  from the noncommutative viewpoint}, Trans. Amer. Math. Soc. \textbf{354}
  (2002), no.~2, 735--747.

\bibitem[GJS23]{GJS19}
Sam Gunningham, David Jordan, and Pavel Safronov, \emph{The finiteness
  conjecture for skein modules}, Invent. Math. \textbf{232} (2023), no.~1,
  301--363.

\bibitem[GL05]{GL05}
Stavros Garoufalidis and Thang T.~Q. L\^e, \emph{The colored {J}ones function
  is {$q$}-holonomic}, Geom. Topol. \textbf{9} (2005), 1253--1293.

\bibitem[HP95]{HP95}
Jim Hoste and J\'ozef~H. Przytycki, \emph{The {K}auffman bracket skein module
  of {$S^1\times S^2$}}, Math. Z. \textbf{220} (1995), no.~1, 65--73.

\bibitem[hyx]{MathOver}
Stanley Yao~Xiao (https://mathoverflow.net/users/10898/stanley-yao xiao),
  \emph{Algebraic numbers with lacunary vanishing polynomial}, MathOverflow,
  \eprint{https://mathoverflow.net/q/484587},
  URL:https://mathoverflow.net/q/484587 (version: 2024-12-21).

\bibitem[JR25]{JR25}
David Jordan and Iordanis Romaidis, \emph{Finiteness and holonomicity of skein
  modules}, To appear (2025).

\bibitem[KW23]{KW23}
Hiroaki Karuo and Zhihao Wang, \emph{Finiteness conjecture for 3-manifolds
  obtained from handlebodies by attaching 2-handles}, arXiv preprint
  arXiv:2401.00262 (2023).

\bibitem[L{\^{e}}06]{Le06}
Thang T.~Q. L{\^{e}}, \emph{The colored {J}ones polynomial and the
  {$A$}-polynomial of knots}, Adv. Math. \textbf{207} (2006), no.~2, 782--804.

\bibitem[LT14]{LT14}
Thang T.~Q. Le and Anh~T. Tran, \emph{The {K}auffman bracket skein module of
  two-bridge links}, Proc. Amer. Math. Soc. \textbf{142} (2014), no.~3,
  1045--1056.

\bibitem[Lus93]{Lust:book}
George Lusztig, \emph{Introduction to quantum groups}, Progress in Mathematics,
  vol. 110, Birkh\"auser Boston, Inc., Boston, MA, 1993.

\bibitem[Mar10]{Mar10}
Julien March\'e, \emph{The skein module of torus knots}, Quantum Topol.
  \textbf{1} (2010), no.~4, 413--421.

\bibitem[Mil63]{Mil:book}
J.~Milnor, \emph{Morse theory}, Annals of Mathematics Studies, vol. No. 51,
  Princeton University Press, Princeton, NJ, 1963, Based on lecture notes by M.
  Spivak and R. Wells.

\bibitem[Mro11]{Mro11}
Maciej Mroczkowski, \emph{Kauffman bracket skein module of the connected sum of
  two projective spaces}, J. Knot Theory Ramifications \textbf{20} (2011),
  no.~5, 651--675.

\bibitem[MV94]{MV94}
G.~Masbaum and P.~Vogel, \emph{{$3$}-valent graphs and the {K}auffman bracket},
  Pacific J. Math. \textbf{164} (1994), no.~2, 361--381.

\bibitem[Prz91]{Prz1}
J\'{o}zef~H. Przytycki, \emph{Skein modules of {$3$}-manifolds}, Bull. Polish
  Acad. Sci. Math. \textbf{39} (1991), no.~1-2, 91--100.

\bibitem[Prz99]{Prz99}
J\'ozef~H. Przytycki, \emph{Fundamentals of {K}auffman bracket skein modules},
  Kobe J. Math. \textbf{16} (1999), no.~1, 45--66.

\bibitem[Tur88]{Tur}
V.~G. Turaev, \emph{The {C}onway and {K}auffman modules of a solid torus}, Zap.
  Nauchn. Sem. Leningrad. Otdel. Mat. Inst. Steklov. (LOMI) \textbf{167}
  (1988), no.~Issled. Topol. 6, 79--89, 190.

\end{thebibliography}

\end{document}